\documentclass[review,onefignum,onetabnum]{siamart171218}

\usepackage{lipsum}
\usepackage{float}
\usepackage{mathrsfs}
\usepackage{amsfonts}
\usepackage{graphicx}
\usepackage{epstopdf}
\usepackage{algorithmic}
\usepackage{amssymb}
\usepackage{amsmath}

\ifpdf
  \DeclareGraphicsExtensions{.eps,.pdf,.png,.jpg}
\else
  \DeclareGraphicsExtensions{.eps}
\fi

\newsiamremark{hypothesis}{Hypothesis}
\crefname{hypothesis}{Hypothesis}{Hypotheses}
\newsiamthm{claim}{Claim}
\newtheorem{remark}{Remark}
\nolinenumbers

\headers{Legendre and Conical functions}{T. M. Dunster}

\title{Simplified uniform asymptotic expansions for associated Legendre and conical functions}

\author{T. M. Dunster\thanks{Department of Mathematics and Statistics, San Diego State University, 5500 Campanile Drive, San Diego, CA 92182-7720, USA. 
  (\email{mdunster@sdsu.edu}, \url{https://tmdunster.sdsu.edu}).}}

\usepackage{amsopn}

\makeatletter
\newcommand*{\addFileDependency}[1]{
  \typeout{(#1)}
  \@addtofilelist{#1}
  \IfFileExists{#1}{}{\typeout{No file #1.}}
}
\makeatother

\nolinenumbers

\ifpdf
\hypersetup{
  pdftitle={Simplified uniform asymptotic expansions for associated Legendre and conical functions},
  pdfauthor={T. M. Dunster}
}

\begin{document}

\maketitle

\begin{abstract}
Asymptotic expansions are derived for associated Legendre functions of degree $\nu$ and order $\mu$, where one or the other of the parameters is large. The expansions are uniformly valid for unbounded real and complex values of the argument $z$, including the singularity $z=1$. The cases where $\nu+\frac12$ and $\mu$ are real or purely imaginary are included, which covers conical functions. The approximations involve either exponential or modified Bessel functions, along with slowly varying coefficient functions. The coefficients of the new asymptotic expansions are simple and readily obtained explicitly, allowing for computation to a high degree of accuracy. The results are constructed and rigorously established by employing certain Liouville-Green type expansions where the coefficients appear in the exponent of an exponential function.
\end{abstract}



\section{Introduction} 
\label{sec1}

The associated Legendre functions of the first and second kinds have the explicit representation (see, for example, \cite[Eqs. 14.3.6, 14.3.7, 14.3.10]{NIST:DLMF})
\begin{equation}
\label{eq1.1}
P^{-\mu}_{\nu}(z)=
\left(\frac{z-1}{z+1}\right)^{\mu/2}
\mathbf{F}\left(\nu+1,-\nu;1+\mu;
\tfrac{1}{2}-\tfrac{1}{2}z\right),
\end{equation}
and
\begin{equation}
\label{eq1.1a}
\boldsymbol{Q}^{\mu}_{\nu}(z)
=\frac{\sqrt{\pi}
\left(z^{2}-1\right)^{\mu/2}}
{2^{\nu+1}z^{\nu+\mu+1}}
\mathbf{F}\left(
\tfrac{1}{2}\nu+\tfrac{1}{2}\mu+1,
\tfrac{1}{2}\nu+\tfrac{1}{2}\mu+\tfrac{1}{2}
;\nu+\tfrac{3}{2};\frac{1}{z^{2}}\right),
\end{equation}
where $\mathbf{F}$ is Olver's scaled hypergeometric function defined by
\begin{equation}
\label{eq1.12}
\textbf{F}(a,b;c;z)=
\frac{F(a,b;c;z)}{\Gamma(c)}
=\sum_{s=0}^{\infty}\frac{(a)_{s}(b)_{s}}
{\Gamma(c+s)}\frac{z^{s}}{s!}.
\end{equation}

Principal branches are taken for $z>1$, and the functions are defined by continuity elsewhere in the plane having a cut along $(-\infty,1]$, and these functions are generally complex-valued on the cut. Ferrers functions are real-valued for $z=x \in (-1,1)$, and are given by (see \cite[Eqs. 14.23.4 and 14.23.5]{NIST:DLMF})
\begin{equation}
\label{eq5.1}
\mathsf{P}_{\nu}^{-\mu}(x)
=e^{\mp \mu \pi i/2}
P_{\nu}^{-\mu}(x \pm i0),
\end{equation}
and
\begin{equation}
\label{eq5.2}
\mathsf{Q}_{\nu}^{-\mu}(x)
=\tfrac12 \Gamma(\nu-\mu+1)\left\{e^{\mu \pi i/2}
\boldsymbol{Q}_{\nu}^{\mu}(x +i0)
+e^{-\mu \pi i/2}\boldsymbol{Q}_{\nu}^{\mu}(x -i0)\right\},
\end{equation}
on using the standard notation $f(x \pm i0)=\lim_{\epsilon \to 0+}f(x \pm i\epsilon)$.

All these functions are solutions of the associated Legendre differential equation
\begin{equation}
\label{eq1.2x}
(z^2-1)\frac{d^2 y}{d z^2} +2z\frac{dy}{dz} 
-\left\{\nu(\nu + 1)+\frac{\mu^2}{z^2-1}\right\}y=0,
\end{equation}
which has regular singularities at $z=\pm 1$ and $z=\infty$.

Legendre functions arise in various problems involving spherical harmonics and potential theory (see \cite[Sec. 14.30]{NIST:DLMF}), and as such the study of these functions, particularly in the context of large parameters, has garnered significant attention.  As the parameters \( \nu \) and \( \mu \) become large, understanding their asymptotic behavior is important for both theoretical insight and computation.

Previous research in this field has resulted in numerous asymptotic approximations and expansions for various parameter regimes. For example, uniform asymptotic approximations for large $\nu$ and fixed $\mu$ were derived in \cite{Boyd:1990:CSP}, \cite{Frenzen:1990:EBU}, \cite{Jones:2001:AHF}, \cite[Chap. 12, Secs. 12 and 13]{Olver:1997:ASF}, \cite{Shivakumar:1988:EBU}, \cite[Chap. 29]{Temme:2015:AMF}, \cite{Ursell:1984:ILP}, and \cite[Chap. VII, Sec. 8]{Wong:1989:AAI}, with convergent series expansions given in \cite{Dunster:2004:CSP}. For large $\mu$ and fixed $\nu$ see \cite{Dunster:2003:ALF}, \cite{Gil:2000:CTF} and \cite[Chap. 29]{Temme:2015:AMF}. More recently, simple (inverse) factorial expansions (with error bounds) are constructed in \cite{Nemes:2020:LPA} for one of the degree or order large, valid for unbounded $z$ but not at or near the pole $z=1$ of (\ref{eq1.2x}).

Conical functions are associated Legendre functions with the complex-valued degree $\nu=-\frac12+i\tau$ ($\tau \in \mathbb{R}$), and have a number of important applications, most notably in solving Laplace's equation expressed in toroidal coordinates \cite[Chap. 7]{Sneddon:1972:UIT}. They also appear in the kernels of the Mehler-Fock transforms \cite[Chap. 7]{Sneddon:1972:UIT} and solutions of the Helmholtz equation in geodesic polar coordinates \cite{Cohl:2018:FSG}. Asymptotic expansions for these functions have been given by \cite{Cohl:2018:FSG}, \cite{Dunster:2013:CFP}, and \cite[pp. 473–474]{Olver:1997:ASF}. Further, in \cite{Dunster:1991:CFW} expansions are given which are valid for one or both of the parameters being allowed to be large, these being uniformly valid in unbounded domains that contain $z=1$.

The case where both parameters are permitted to be large is generally more complicated as it involves coalescing turning points and poles. Asymptotic approximations and expansions for associated Legendre and Ferrers functions were constructed in such cases by \cite{Boyd:1986:TPS}, \cite{Dunster:2003:ALF} and \cite{Olver:1975:LFW}, for large $\mu$ with $\nu/\mu \in [0,1-\delta]$, large $\mu$ with $\nu/\mu \in [0,1-\delta]$, and large $\nu$ with $\mu/\nu \in [\delta,\delta^{-1}]$, respectively. \emph{Here and throughout this paper $\delta$ represents a generic arbitrarily small positive constant.}

Existing asymptotic expansions typically have coefficients that are not easy to evaluate, or are valid for a restricted range of the argument, such as $z$ being bounded away from $1$ or lying in a bounded domain. An exception is the case where $\mu \to \infty$ and $z=x \in [0,1)$, in which Liouville-Green (LG) expansions are valid for small to large $|\nu|$ (see \cite[Sec. 5]{Dunster:1991:CFW} and \cite[Sec. 4]{Dunster:2003:ALF}) and have coefficients which are straightforward to evaluate. The reason in this case is that the differential equation in question is free from turning points (where LG expansions break down) and the singularities of the approximating differential equation involves a double pole at the end points $x=\pm 1$, as opposed to a simple pole which is the case for the other parameter regimes: under appropriate conditions LG expansions are valid at the former but generally are not so at the latter. This will be expanded upon further below.

In this paper we provide new asymptotic expansions for associated Legendre, conical and Ferrers functions with one of the parameters large and the other bounded, with the exceptional case as described above where both can be large. We consider real and complex values of the argument $z$, in both cases our expansions are uniformly valid at the singularity $z=1$, and in the complex case also for $z$ unbounded. The significance is that our coefficients appear in the arguments of certain exponential, hyperbolic or trigonometric functions, which results in simple expressions for them which only involve recursively defined polynomials and the LG variable. Another advantage, as shown in \cite[Sec. 5]{Dunster:2024:AEG}, is that construction of error bounds only requires the use of the relatively simple ones for LG type expansions constructed in \cite{Dunster:2020:LGE}. In particular error bounds using this method does not require estimates involving the higher functions (here modified Bessel functions) that appear in the expansions.

\emph{In this paper we assume $\Re(\mu) \geq 0$ and $\Re(\nu) \geq -\frac12$}. Extensions to other parameter ranges are easily obtained via appropriate connection formulas (see for example \cite[Sec. 14.9]{NIST:DLMF}). Also, we consider the half plane $|\arg(z)|\leq \frac12 \pi$, or in the real variable case $z=x \in (-1,1)$ (in conjunction with connection formulas for negative $x$ in the large $\nu$ case). For complex $z$ extension to the left half-plane and other sheets follows readily from \cite[Eqs. 14.24.1	and 14.24.2]{NIST:DLMF}  
\begin{equation}
\label{eq2.22b}
P^{-\mu}_{\nu}\left(ze^{s\pi i}\right)
=e^{s\nu\pi i}P^{-\mu}_{\nu}(z)
+\frac{2i\sin\left(\left(\nu+\frac{1}{2}\right)
s\pi\right)e^{-s\pi i/2}
}{\cos(\nu\pi)\Gamma(\mu-\nu)}
\boldsymbol{Q}^{\mu}_{\nu}(z),
\end{equation}
and
\begin{equation}
\label{eq2.22x}
\boldsymbol{Q}^{\mu}_{\nu}
\left(ze^{s\pi i}\right)
=(-1)^{s}e^{-s\nu\pi i}
\boldsymbol{Q}^{\mu}_{\nu}(z).
\end{equation}

In addition to the associated Legendre and Ferrers functions defined above, we also will approximate analytic continuations of $\boldsymbol{Q}^{\mu}_{\nu}(z)$ denoted by
\begin{equation}
\label{eq2.22y}
\boldsymbol{Q}^{\mu}_{\nu,\pm 1}(z)
=\boldsymbol{Q}^{\mu}_{\nu}
\left(1+(z-1)e^{\pm 2\pi i}\right),
\end{equation}
i.e. the branches of $\boldsymbol{Q}^{\mu}_{\nu}(z)$ obtained from the principal branch of the function by encircling the branch point $z=1$ (but not the branch point $z=-1$) in the positive (respectively negatively) sense. From \cite[Eq. 14.24.4]{NIST:DLMF} these can be expressed in terms of the other functions by
\begin{equation}
\label{eq2.22e}
\boldsymbol{Q}^{\mu}_{\nu,\pm 1}(z)
=e^{\mp \mu\pi i}\boldsymbol{Q}^{\mu}_{\nu}(z)
\mp \frac{\pi i}{\Gamma(\nu-\mu+1)}
P^{-\mu}_{\nu}(z),
\end{equation}
which in turn implies
\begin{equation}
\label{eq2.22ef}
\cos(\mu \pi)P^{-\mu}_{\nu}(z)
=\tfrac12 \Gamma(\nu-\mu+1)
\left\{e^{(\mu+\frac12 )\pi i}
\boldsymbol{Q}^{\mu}_{\nu,1}(z)
-e^{-(\mu+\frac12 ) \pi i}
\boldsymbol{Q}^{\mu}_{\nu,-1}(z)\right\}.
\end{equation}

We also have the relations
\begin{equation}
\label{eq2.22m}
\frac{2\sin(\mu\pi)}{\pi}\boldsymbol{Q}^{\mu}_{\nu}(z)
=\frac{P^{\mu}_{\nu}(z)}
{\Gamma\left(\nu+\mu+1\right)}
-\frac{P^{-\mu}_{\nu}(z)}
{\Gamma\left(\nu-\mu+1\right)},
\end{equation}
\begin{equation}
\label{eq2.22f}
\cos(\mu \pi)\boldsymbol{Q}^{\mu}_{\nu}(z)
=\tfrac12\left\{\boldsymbol{Q}^{\mu}_{\nu,1}(z)
+\boldsymbol{Q}^{\mu}_{\nu,-1}(z)\right\},
\end{equation}
and
\begin{equation}
\label{eq2.22g}
\frac{2\pi i \cos(\mu \pi)}{\Gamma(\nu+\mu+1)}
P^{\mu}_{\nu}(z)
=e^{\mu \pi i}\boldsymbol{Q}^{\mu}_{\nu,-1}(z)
-e^{-\mu \pi i}\boldsymbol{Q}^{\mu}_{\nu,1}(z).
\end{equation}

The fundamental behaviour of these functions at the singularities is given by (see for example \cite[Eqs. 14.8.7 and 14.8.15]{NIST:DLMF})
\begin{equation}
\label{eq1.1b}
P^{-\mu}_{\nu}(z)=
\frac{1}{\Gamma(\mu+1)}
\left(\frac{z-1}{2}\right)^{\mu/2}
\left\{1+\mathcal{O}(z-1)\right\}
\quad (z \rightarrow 1),
\end{equation}
and
\begin{equation}
\label{eq1.1c}
\boldsymbol{Q}^{\mu}_{\nu}(z)
=\frac{\sqrt{\pi}}
{\Gamma\left(\nu+\frac{3}{2}\right)
(2z)^{\nu+1}}
\left\{1+\mathcal{O}
\left(\frac{1}{z}\right)\right\}
\quad \left(z \rightarrow \infty, \,
\nu \neq -\tfrac{3}{2}, -\tfrac{5}{2}, -\tfrac{7}{2}, \ldots \right);
\end{equation}
for the exceptional cases see \cite[Eq. 14.8.16]{NIST:DLMF}. These limiting forms show that the two functions are recessive at the respective singularities, and they form a numerically satisfactory pair of solutions of (\ref{eq1.2x}) in the half-plane $|\arg(z)|\leq \frac12 \pi$ when $\Re(\mu) \geq 0$ and $\Re(\nu) \geq -\frac12$ (see \cite[Chap. 5, Thm. 12.1]{Olver:1997:ASF}).

Consider now the asymptotic behaviour of solutions of (\ref{eq1.2x}) as $\nu \to \infty$. To do so we follow \cite[Chap. 12, Secs. 12 and 13]{Olver:1997:ASF} and let
\begin{equation}
\label{eq2.1}
w(z)=\left(z^2-1\right)^{1/2}y(z),
\end{equation}
and
\begin{equation}
\label{eq2.2}
u=\nu+\tfrac{1}{2}.
\end{equation}
Then (\ref{eq1.2x}) becomes
\begin{equation}
\label{eq2.3}
w''(z)=\left\{u^2f(z)+g(\mu,z)\right\}w(z),
\end{equation}
where
\begin{equation}
\label{eq2.4}
f(z)=\frac{1}{z^2-1}, \;
g(\mu,z)=-\frac{1}{4 \left ( z^2-1 \right )}
+\frac{\mu^2 -1}
{\left ( z^2-1 \right )^2}.
\end{equation}
Here the choice of large parameter $u=\nu+\frac12$, as opposed to just simply $\nu$, results in the correct form of the accompanying function $g(\mu,z)$, which then ensures that the subsequent asymptotic approximations are valid at $z= \infty$; see \cite[Chap. 10, Ex. 4.1]{Olver:1997:ASF}.

The function $f(z)$, which is the principal term in approximations when $u$ is large, has simple poles at $z= \pm 1$, and typically LG expansions break down at such points (see \cite[Chap. 10, Sec. 4]{Olver:1997:ASF}). Instead, expansions for Legendre functions when $u$ is large that are valid at the pole are provided by \cite[Chap. 12, Secs. 12 and 13]{Olver:1997:ASF}, and these involve modified Bessel functions. The disadvantage is that the coefficients in the expansions are hard to compute since they are not explicitly given, and instead require nested integrations in their evaluation.

Instead we use the method and some results of \cite{Dunster:2024:AEG} in which the closely related Gegenbauer  and companion functions were studied. In that paper, similar to \cite{Boyd:1990:CSP} and \cite{Dunster:2004:CSP}, exact solutions involving modified Bessel functions and two slowly-varying coefficient functions were defined. The difference in \cite{Dunster:2024:AEG} as compared to \cite{Boyd:1990:CSP} and \cite{Dunster:2004:CSP} is that LG expansions of exponential form were used to construct asymptotic expansions for the coefficient functions, and these in turn involve coefficients in their asymptotic expansions that are straightforward to evaluate explicitly. Indeed a similar method was first employed for turning point expansions involving Airy functions in \cite{Dunster:2017:COA}, with error bounds given in \cite{Dunster:2021:SEB}. In those papers Cauchy's integral formula was used to evaluate the approximations in a full neighborhood of a turning point, as well as bound the error terms there, and similarly in \cite{Dunster:2024:AEG} in a neighborhood of a simple pole.

For brevity explicit error bounds are not presented here, although they can readily be constructed using those given in \cite[Sec. 5]{Dunster:2024:AEG} which utilised the LG error bounds given in \cite{Dunster:2020:LGE}. Indeed these bounds rigorously establish the uniform validity of all the expansions in this paper, without having to resort to the more complicated error analysis involving so-called weight, modulus and phase functions for Bessel functions (see \cite[Chap. 12, Secs. 1 and 8]{Olver:1997:ASF}).

The plan of this paper is as follows. In \cref{sec.LargeNu} we obtain expansions for $\nu$ large and positive, with $\mu$ bounded and real or complex. These are obtained from identifying associated Legendre functions with the Gegenbauer and related functions given in \cite{Dunster:2024:AEG}. The extension to the conical case $\nu =-\frac12+i \tau$ with $\tau \to \infty$ is considered in \cref{sec.LargeTau}, with again with $\mu$ bounded and real or complex, and expansions for $\mu$ large and positive, with $\nu$ bounded and real or complex (including conical functions) are derived in \cref{sec.LargeMu}. In all these cases the expansions are uniformly valid for $|\arg(z)|\leq \frac12 \pi$, $|z-1| \geq \delta>0$, with extension to this excluded neighbourhood of the pole at $z=1$ achieved by Cauchy's integral formula as described above.

The case $z=x \in (-1,1)$ is tackled in \cref{sec.Ferrers}. Firstly, expansions for Ferrers and conical functions are derived that are uniformly valid in the interval $[0,1)$ for $|\nu|$ large, with $\mu \geq 0$ bounded; extension to $(-1,0)$ follows from appropriate well-known connection formulas. Then we consider $\mu \to \infty$, both for $\nu \in \mathbb{R}$ such that $0 \leq \mu/\nu \leq 1-\delta$, and $\nu = -\frac12 + i \tau$ with $\tau >0$ and $0 \leq |\mu/\nu| \leq \alpha_{0}$ ($0<\alpha_{0}<\infty$). 

Finally, some numerical results to illustrate the accuracy of our expansions (of real and complex arguments) are presented in \cref{sec.numerics}.

\section{Large positive \texorpdfstring{$\nu$}{} with \texorpdfstring{$\mu$}{} bounded}
\label{sec.LargeNu}

From \cite[Eq. 14.3.22]{NIST:DLMF}
\begin{equation}
\label{eq1.1e}
P^{-\mu}_{\nu}(z)=\frac{\Gamma\left(2\mu+1\right)
\Gamma\left(\nu-\mu+1\right)\left(z^{2}-1\right)^{\mu/2}}
{2^{\mu}\Gamma\left(\mu+1\right)
\Gamma\left(\nu+\mu+1\right)}
C^{(\mu+\frac{1}{2})}_{\nu-\mu}(z),
\end{equation}
where $C_{n}^{(\lambda)}(z)$ is the Gegenbauer function which can be expressed in terms of the hypergeometric function \cite[Eqs. 15.9.2, 18.5.10]{NIST:DLMF}. When $n$ is a non-negative integer it reduces to the Gegenbauer (or ultraspherical) polynomial 
\begin{equation*}
\label{eq2.1a}
C^{(\lambda)}_{n}(z)=
\sum_{k=0}^{\left \lfloor n/2\right \rfloor}(-1)^{k}
\frac{\Gamma (n-k+\lambda )}
{\Gamma (\lambda )k!(n-2k)!}(2z)^{n-2k}
\quad (n=0,1,2,\ldots).
\end{equation*}
On account of (\ref{eq1.1e}) and \cite[Eqs. (1.1), (1.9) and (2.2)]{Dunster:2024:AEG} we replace in that paper $\nu \mapsto \mu$, $\lambda \mapsto \mu+\frac12$, $n \mapsto \nu - \mu$ and $u \mapsto \nu +\frac12$.

In terms of a companion function $D_{n}^{(\lambda)}(z)$ given by \cite[Eqs. (1.12) and (1.13)]{Dunster:2024:AEG} we also have from \cite[Eq. 14.3.19]{NIST:DLMF} the relationship
\begin{equation}
\label{eq1.1f}
\boldsymbol{Q}^{\mu}_{\nu}(z)
=\frac{2^{\mu}\Gamma\left(\mu+\tfrac12\right)
\left(z^{2}-1\right)^{\mu/2}}
{\sqrt{\pi}\Gamma\left(\nu+\mu+1\right)}
D^{(\mu+\frac{1}{2})}_{\nu-\mu}(z).
\end{equation}

From \cite[Eqs. (2.4) and (2.13)]{Dunster:2024:AEG} (see also (\ref{eq2.4})) we next define the Liouville variable by
\begin{equation}
\label{eq2.6}
\xi =\int_{1}^{z} f^{1/2}(t) dt
=\ln\left\{z+(z^2-1)^{1/2}\right\}
=\mathrm{arccosh}(z),
\end{equation}
which plays a central role in the subsequent asymptotic expansions. The branches are such that $\xi>0$ for $z>0$ with $\arg(z)=0$, and by continuity elsewhere in the $z$ plane having a cut along $(-\infty,1]$. Now it is straightforward to show that
\begin{equation}
\label{eq2.6aa}
\xi =\sqrt{2} (z - 1)^{1/2}- \tfrac{1}{12}\sqrt{2}(z - 1)^{3/2}
+\mathcal{O}\left\{(z - 1)^{5/2}\right\}
\quad (z \to 1),
\end{equation}
and
\begin{equation}
\label{eq2.6a}
\xi =\ln(2z)-\tfrac{1}{4}z^{-2}
+\mathcal{O}\left(z^{-4}\right)
\quad (z \to \infty).
\end{equation}

\begin{figure}[h!]
 \centering
 \includegraphics[
 width=0.7\textwidth,keepaspectratio]{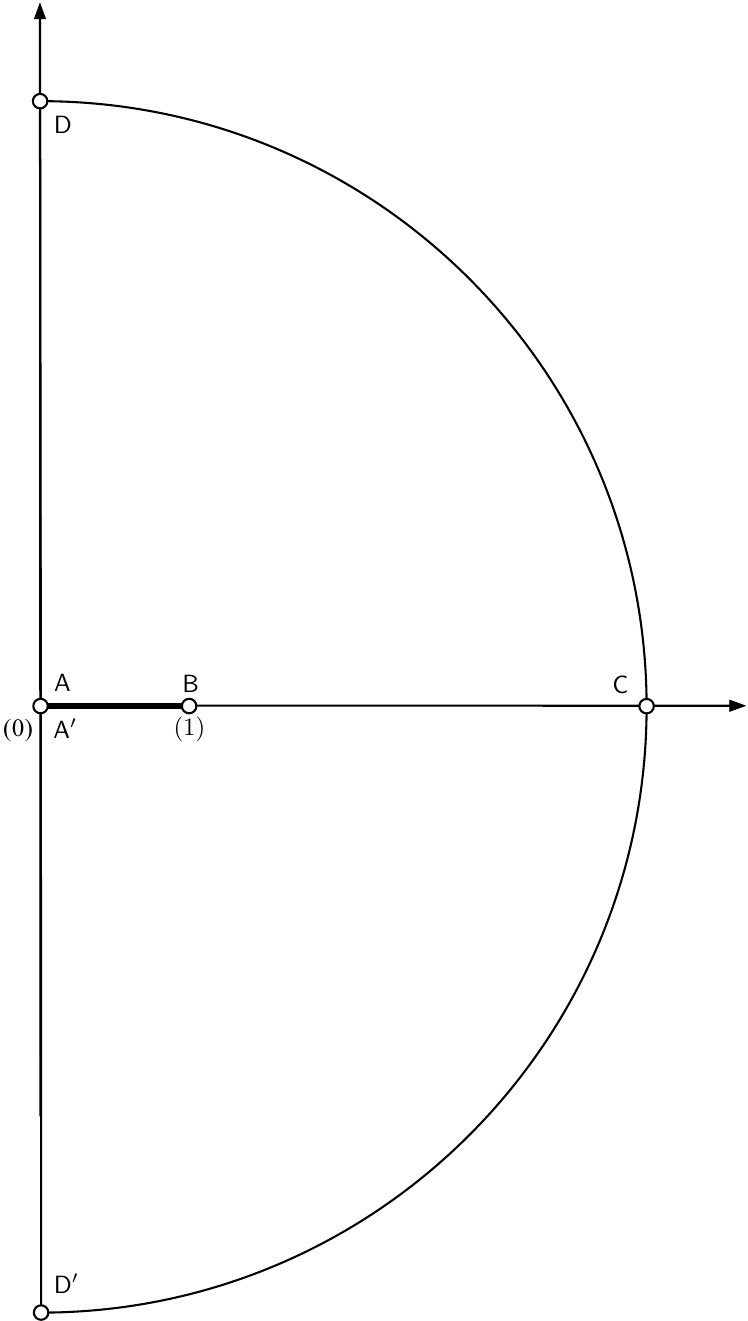}
 \caption{$z$ plane.}
 \label{fig:Figzplane}
\end{figure}

\begin{figure}[h!]
 \centering
 \includegraphics[
 width=\textwidth,keepaspectratio]{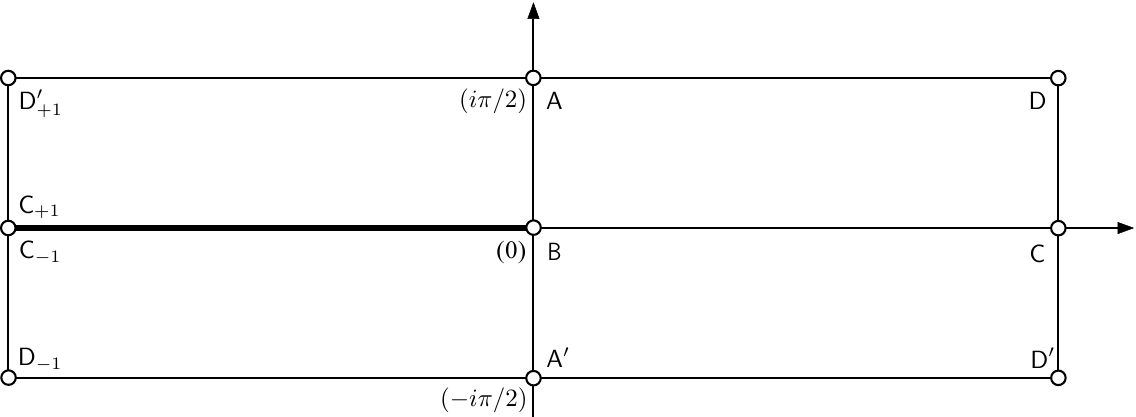}
 \caption{$\xi$ plane.}
 \label{fig:Figxiplane}
\end{figure}

In due course we shall need a more detailed description of the $z-\xi$ map, which is provided as follows. The half-plane $|\arg(z)|\leq \frac12 \pi$ with a cut along $[0,1]$ is depicted in \cref{fig:Figzplane}, with corresponding points in the $\xi$ plane shown in \cref{fig:Figxiplane}. In the former, the points $\mathsf{C}$, $\mathsf{D}$ and $\mathsf{D}'$ lie on a semi-circle of arbitrary large radius, so that the corresponding points in the $\xi$ plane have arbitrary large and positive real part, with $|\Im(\xi)| \leq \frac12 \pi$; see (\ref{eq2.6a}).

Also in \cref{fig:Figxiplane} the point $\mathsf{C}_{+1}$ corresponds to the point $\mathsf{C}$ on the sheet in the $z$ plane when accessed across the cut from above (i.e. from encircling the branch point $z=1$ in a positive sense). Likewise $\mathsf{C}_{-1}$ in the $\xi$ plane corresponds to the point $\mathsf{C}$ on the sheet in the $z$ plane when accessed from below the cut, and similarly for $\mathsf{D}_{+1}'$ and $\mathsf{D}_{-1}$. Due to the corresponding points in the $z$ plane being arbitrarily large these are considered to have real parts that are arbitrarily large in absolute value.

In summary, the half-plane $|\arg(z)|\leq \frac12 \pi$ is mapped to the strip $0 \leq \Re(\xi) < \infty$, $|\Im(\xi)| \leq \frac12 \pi$. Moreover, the first quadrant on the sheet in $z$ plane accessed by crossing the cut $[0,1]$ from below is mapped to the the strip $-\infty < \Re(\xi) \leq 0$, $-\frac12 \pi \leq \Im(\xi) \leq 0$, and the fourth quadrant on the sheet in $z$ plane accessed by crossing the cut $[0,1]$ from above is mapped to the the strip $-\infty < \Re(\xi) \leq 0$, $0 \leq \Im(\xi) \leq \frac12 \pi$.

Next we define the coefficients that appear in our expansions. From \cite[Eqs. (2.18) - (2.24)]{Dunster:2024:AEG} let
\begin{equation}
\label{eq2.12}
\beta=z \left(z^2 - 1\right)^{-1/2},
\end{equation}
and then define $F_{\mu,s}(\beta)$ ($s=1,2,3,\ldots$) as polynomials in $\beta$ given recursively by
\begin{equation}
\label{eq2.14}
F_{\mu,1}(\beta )
=\tfrac{1}{8}\left (4\mu^2-1 \right)\left (\beta^2-1 \right),
\end{equation}
\begin{equation}
\label{eq2.15}
F_{\mu,2}(\beta)=\tfrac{1}{8}
\left( 4\mu^2-1\right)\beta\left(\beta^2-1\right),
\end{equation}
and
\begin{equation}
\label{eq2.16}
F_{\mu,s+1}(\beta)
=\frac{1}{2}\left(\beta^2-1\right)
\frac{dF_{\mu,s}(\beta)}{d\beta}-{\frac{1}{2}}
\sum\limits_{j=1}^{s-1}
F_{\mu,j}(\beta)F_{\mu,s-j}(\beta)
\quad (s=2,3,4,\ldots).
\end{equation}
Each $F_{\mu,s}(\beta)$ is divisible by $\beta^2-1$, as can be seen from (\ref{eq2.14}) - (\ref{eq2.16}) and using induction. 

Note from (\ref{eq2.4}), (\ref{eq2.6}) and (\ref{eq2.12}) that
\begin{equation}
\label{eq2.13a}
\xi=\frac12 \ln\left(
\frac{\beta+1}{\beta-1}\right),
\end{equation}
and
\begin{equation}
\label{eq2.13}
\frac{d\xi}{d\beta}
=\frac{d\xi}{d z}\left(\frac{d \beta}{d z}\right)^{-1}
=-\frac{1}{\beta^2-1}.
\end{equation}
Then from (\ref{eq2.13}) and \cite[Eq. (1.10)]{Dunster:2020:LGE} the required LG coefficients $E_{\mu,s}(\beta)$ ($s=1,2,3,\ldots$), which are also polynomials in $\beta$, are given by (c.f. \cite[Eq. (2.24)]{Dunster:2024:AEG})
\begin{equation}
\label{eq2.17}
E_{\mu,s}(\beta) =-\int_{0}^{\beta}
\frac{F_{\mu,s}(b)}{b^2-1} d b
\quad (s=1,2,3,\ldots).
\end{equation}
Here we have chosen the arbitrary integration constants so that $E_{s}(0)=0$. Observe that $E_{\mu,2s}(\beta)$ are even and $E_{\mu,2s+1}(\beta)$ are odd functions of $\beta$. The latter property means that, regarded as functions of $z$ via (\ref{eq2.12}), each $(z-1)^{1/2}E_{\mu,2s-1}(\beta)$ ($s=1,2,3,\ldots$) is a meromorphic function at $z=1$, and this is a requirement for our expansions to be valid in a punctured neighbourhood of $z=1$.

The first two coefficients are given by
\begin{equation}
\label{eq2.17a}
E_{\mu,1}(\beta) =-\tfrac18
\left(4\mu^2-1\right) \beta, \;
E_{\mu,2}(\beta) =-\tfrac{1}{16}
\left(4\mu^2-1\right) \beta^2.
\end{equation}

We are now in a position to state the main result of this section.

\begin{theorem}
\label{thm:nularge}

Let $u=\nu+\frac12$, $\xi$ be given by (\ref{eq2.6}), and
\begin{equation}
\label{eq3.24L}
L_{\nu}^{\mu}=
\frac{\sqrt{(2\nu+1)\pi}}
{2^{\nu+1}\Gamma \left(\frac12 \nu + \frac12 \mu +1\right)
\Gamma \left(\frac12 \nu - \frac12 \mu +1\right)}.
\end{equation}
Next, in terms of the modified Bessel functions \cite[Sec. 10.25]{NIST:DLMF}, let
\begin{multline}
\label{eq3.24e}
P_{\nu}^{-\mu}(z)
=L_{\nu}^{\mu} \, \Gamma(\nu-\mu+1)
 \left(z^2-1\right)^{-1/4}
\\ \times
 \xi^{1/2} \left\{I_{\mu}(u \xi)A_{\mu}(u,z)
+\xi I_{\mu+1}(u \xi)B_{\mu}(u,z)\right\},
\end{multline}
\begin{equation}
\label{eq3.24ee}
\boldsymbol{Q}^{\mu}_{\nu}(z)
=L_{\nu}^{\mu} \, 
\left(z^2-1\right)^{-1/4} \xi^{1/2}
\left\{K_{\mu}(u \xi)A_{\mu}(u,z)
-\xi K_{\mu+1}(u \xi)B_{\mu}(u,z)\right\},
\end{equation}
\begin{multline}
\label{eq3.25}
P_{\nu}^{\mu}(z)
=L_{\nu}^{\mu} \, \Gamma(\nu+\mu+1)
 \left(z^2-1\right)^{-1/4}
\\ \times
 \xi^{1/2} \left\{I_{-\mu}(u \xi)A_{\mu}(u,z)
+ \xi I_{-\mu-1}(u \xi)B_{\mu}(u,z)\right\},
\end{multline}
and
\begin{multline}
\label{eq9.1a}
\boldsymbol{Q}^{\mu}_{\nu,\pm 1}(z)
=L_{\nu}^{\mu} \, 
\left(z^2-1\right)^{-1/4} 
\\ \times
\xi^{1/2}
\left\{K_{\mu}(u \xi e^{\pm \pi i})A_{\mu}(u,z)
+\xi K_{\mu+1}(u \xi e^{\pm \pi i})B_{\mu}(u,z)\right\},
\end{multline}
where the fractional powers take positive values when $z>1$ ($\xi >0$) and are defined by continuity elsewhere. Then $A_{\mu}(u,z)$ and $B_{\mu}(u,z)$ are analytic for $|\arg(z)| \leq \frac12 \pi$, are real for $\nu,\mu \in \mathbb{R}$ and $z \in (-1,\infty)$, and for $\mu$ bounded with $\Re(\mu) \geq 0$ possess the asymptotic expansions 
\begin{equation}
\label{eq3.24f}
A_{\mu}(u,z) \sim 
\exp \left\{
\sum\limits_{s=1}^{\infty}
\frac{\tilde{\mathcal{E}}_{\mu,2s}(z)}{u^{2s}}
\right\}
\cosh \left\{ \sum\limits_{s=0}^{\infty }
\frac{\tilde{\mathcal{E}}_{\mu,2s+1}(z)}{u^{2s+1}}
\right\},
\end{equation}
and
\begin{equation}
\label{eq3.24g}
B_{\mu}(u,z) \sim 
\frac{1}{\xi}\exp \left\{
\sum\limits_{s=1}^{\infty}
\frac{\mathcal{E}_{\mu,2s}(z)}{u^{2s}}
\right\}
\sinh \left\{ \sum\limits_{s=0}^{\infty }
\frac{\mathcal{E}_{\mu,2s+1}(z)}{u^{2s+1}}
\right\},
\end{equation}
as $u \rightarrow \infty$, uniformly for $|\arg(z)| \leq \frac12 \pi$ such that $|z-1| \geq \delta>0$. Here, for $s=1,2,3,\ldots$,
\begin{equation}
\label{eq3.34}
\mathcal{E}_{\mu,s}(z)
=E_{\mu,s}(\beta)+(-1)^{s+1}\frac{a_{\mu,s}}
{s\xi^{s}},
\end{equation}
\begin{equation}
\label{eq3.34a}
\tilde{\mathcal{E}}_{\mu,s}(z)
=E_{\mu,s}(\beta)+(-1)^{s+1}\frac{a_{\mu+1,s}}
{s\xi^{s}},
\end{equation}
in which $E_{\mu,s}(\beta)$ are polynomials in $\beta$ defined by (\ref{eq2.12}) - (\ref{eq2.16}) and (\ref{eq2.17}),
\begin{equation} 
\label{eq2.31}
a_{\mu,1} =  a_{\mu,2}
=\tfrac{1}{8}
\left(4\mu^{2}-1\right),
\end{equation}
and
\begin{equation} 
\label{eq2.32}
a_{\mu,s+1}=\frac{1}{2}(s+1)a_{\mu,s}
-\frac{1}{2}
\sum_{j=1}^{s-1}
a_{\mu,j}a_{\mu,s-j}.
\end{equation}
\end{theorem}

\begin{remark}
The coefficients $a_{\mu,s}$ appear in the expansion
\begin{equation} 
\label{eq2.30}
K_{\mu}(z)\sim
\left(\frac{\pi}{2z}\right)^{1/2}
\exp\left\{-z-\sum_{s=1}^{\infty} (-1)^s
\frac{a_{\mu,s}}{s z^{s}}
\right\}
\quad   \left(z \rightarrow \infty, \,
|\arg(z)| \leq \tfrac{3}{2} \pi - \delta\right),
\end{equation}
which was derived in \cite[Sect. 2.1]{Dunster:2024:AEG} (we use a slightly different notation here).
\end{remark}

\begin{proof}
We start by identifying $\nu \mapsto \mu$, $\lambda \mapsto \mu+\frac12$, $n \mapsto \nu - \mu$, $u \mapsto \nu +\frac12$, $a_{s}(\nu) \mapsto a_{\mu,s}$ and $\tilde{E}_{s}(\beta) \mapsto E_{\mu,s}(\beta)$ from \cite{Dunster:2024:AEG}. Next, noting that the parity of $E_{\mu,s}(\beta)$ matches that of $s$, we have from \cite[Eq. (2.35)]{Dunster:2024:AEG}
\begin{multline}
\label{eq2.27}
\frac{2^{u}\Gamma \left(\frac {1}{2}u
+\frac {1}{2}\mu+\frac34\right)
\Gamma \left(\frac {1}{2}u
-\frac {1}{2}\mu+\frac34\right)}
{\sqrt{\pi}\,\Gamma (u+1)}
\sim 
\exp \left\{
\sum\limits_{s=1}^{\infty}(-1)^{s}
\frac{E_{\mu,s}(1)}{u^{s}}
\right\} 
\\ 
=\exp \left\{
\sum\limits_{s=1}^{\infty}
\frac{E_{\mu,s}(-1)}{u^{s}}
\right\} 
\quad (u \rightarrow \infty,
\; |\arg(u)|\leq \pi-\delta<\pi).
\end{multline}
In \cite{Dunster:2024:AEG} this was proven for $u>0$, but our extension here to $|\arg(u)|\leq \pi-\delta$ follows from referring to Stirling's series \cite[Chap. 8, Eq. (4.04)]{Olver:1997:ASF} for the asymptotics of $\Gamma(z)$ as $z \to \infty$ holds for $|\arg(z)|\leq \pi-\delta$, along with uniqueness of asymptotic expansions \cite[Chap. 1, Sect. 7.2]{Olver:1997:ASF}. Consequently, from (\ref{eq2.2}), (\ref{eq3.24L})  and (\ref{eq2.27}) 
\begin{multline}
\label{eq2.27b}
L_{\nu}^{\mu} \sim 
\frac{\sqrt{u}}{\Gamma(u+1)} 
\exp \left\{\sum_{s=1}^{\infty}(-1)^{s+1} \frac{E_{\mu,s}(1)}{u^s}
\right\}
\\
\left( u=\nu +\tfrac{1}{2} \to \infty, \, 
|\arg(u)| \leq \pi - \delta < \pi \right).
\end{multline}

Now replace the corresponding coefficients $A(u,z)$ and $B(u,z)$ of \cite{Dunster:2024:AEG} by
\begin{equation}
\label{eq3.24p}
A_{\mu}(u,z) 
=\frac{\sqrt{u}\,A(u,z)}{\Gamma(u+1)L_{\nu}^{\mu}},
\end{equation}
and
\begin{equation}
\label{eq3.24q}
B_{\mu}(u,z)
=\frac{\sqrt{u}\,B(u,z)}{\Gamma(u+1)L_{\nu}^{\mu} \,\xi}.
\end{equation}
Then (\ref{eq3.24e}) and (\ref{eq3.24ee}) follow from (\ref{eq1.1e}), (\ref{eq1.1f}), (\ref{eq3.24p}), (\ref{eq3.24q}) and \cite[Eqs. (3.12) and (3.14)]{Dunster:2024:AEG}, and (\ref{eq9.1a}) is obtained on using (\ref{eq2.22e}), (\ref{eq3.24e}), (\ref{eq3.24ee}) and \cite[Eq. 10.34.2]{NIST:DLMF}. Furthermore, (\ref{eq3.25}) follows from (\ref{eq2.22g}), (\ref{eq9.1a}) and (\ref{eq2.50c}) and the Bessel connection formula (see \cite[Eqs. 10.27.3 and 10.34.3]{NIST:DLMF})
\begin{equation}
\label{eq2.50c}
2 \pi i \cos(\mu \pi) I_{-\mu}(z)
=e^{\mu \pi i}K_{\mu}(z e^{-\pi i})
-e^{-\mu \pi i}K_{\mu}(z e^{\pi i}).
\end{equation}
The expansions (\ref{eq3.24f}) and (\ref{eq3.24g}) follow from (\ref{eq2.27b}) - (\ref{eq3.24q}),  and \cite[Eqs. (3.27) and (3.28)]{Dunster:2024:AEG}.

Next, to prove analyticity of the coefficient functions, we have from (\ref{eq3.24e}), (\ref{eq3.25}) and the Wronskian of modified Bessel functions \cite[Eq. 10.28.1]{NIST:DLMF}
\begin{multline} 
\label{eq3.37}
\sin(\mu \pi) A_{\mu}(u,z) = 
\frac{ \pi u \left(z^2-1\right)^{1/4}  \xi^{1/2}}
{2 L_{\nu}^{\mu}}
\left\{ \frac{1}{\Gamma(\nu + \mu +1)}
P^{\mu}_{\nu}(z) I_{\mu+1}(u\xi)
\right.
\\
\left.
- \frac{1}{\Gamma(\nu - \mu +1)}
P^{-\mu}_{\nu}(z) I_{-\mu-1}(u\xi)
\right\},
\end{multline}
and
\begin{multline} 
\label{eq3.38}
\sin(\mu \pi) B_{\mu}(u,z) = 
\frac{ \pi u \left(z^2-1\right)^{1/4}}
{2 L_{\nu}^{\mu} \, \xi^{1/2}}
\left\{ \frac{1}{\Gamma(\nu - \mu +1)}
P^{-\mu}_{\nu}(z) I_{-\mu}(u\xi)
\right.
\\
\left.
- \frac{1}{\Gamma(\nu + \mu +1)}
P^{\mu}_{\nu}(z) I_{\mu}(u\xi)
\right\}.
\end{multline}
We can now expand the RHS of both (\ref{eq3.37}) and (\ref{eq3.38}) as a Taylor series about $z=1$, with the aid of (\ref{eq2.6}) and \cite[Eqs. 10.25.2, 14.3.9 and 15.2.2]{NIST:DLMF} (see also (\ref{eq2.6aa})). These expansions show that the functions $A_{\mu}(u,z)$ and $B_{\mu}(u,z)$ are analytic at $z=1$ ($\xi =0$) for $\sin(\mu \pi) \neq 0$, and by extension for $|\arg(z)| \leq \frac12 \pi$ (including $z=1$). By a limiting argument the restriction $\sin(\mu \pi) \neq 0$ can be relaxed. These expansions also verify that both coefficient functions are real for $\nu,\mu \in \mathbb{R}$ and $z \in (-1,\infty)$.
\end{proof}

\begin{remark}
In \cite{Dunster:2024:AEG} it is stated that the corresponding function $B(u,z)$ is analytic at $z=1$, but it should read that $\xi^{-1}B(u,z)$ (alternatively $(z-1)^{-1/2}B(u,z)$) is analytic at $z=1$ and real for $z \in (-1, \infty)$; apart from the parameters' notation, the only difference in the function $B_{\mu}(u,z)$ here is our scaling factor of (\ref{eq3.24q}), which includes the term $\xi^{-1}$. All the asymptotic expansions in that paper remain valid as stated, and analyticity at $z=1$ is not required but merely chosen here for convenience. The only exception is that \cite[Eq. (5.30)]{Dunster:2024:AEG} should be modified so that $B(u,z)$ and $B_{N}(u,t)$ are replaced by $(z-1)^{-1/2}B(u,z)$ and $(t-1)^{-1/2}B_{N}(u,t)$, respectively, since analyticity is required to apply Cauchy's integral formula.
\end{remark}

To illustrate analyticity we find from (\ref{eq1.1b}), (\ref{eq2.6aa}), (\ref{eq3.24L}), (\ref{eq3.37}), (\ref{eq3.38}) and \cite[Eq. 10.25.2]{NIST:DLMF}
\begin{equation*}
\label{eq3.24l}
A_{\mu}(u,z) = 
M_{\nu}^{\mu}+\mathcal{O}(z-1),
\end{equation*}
and
\begin{equation*}
\label{eq3.24m}
B_{\mu}(u,z) = \frac{u M_{\nu}^{\mu}}{2\mu}
-\frac{u M_{\nu}^{-\mu}}{2\mu}
+\mathcal{O}(z-1),
\end{equation*}
as $z \to 1$, where (on referring to \cite[Eq. 5.11.13]{NIST:DLMF})
\begin{multline*}
\label{eq3.24n}
M_{\nu}^{\mu}=
\left(\frac{2}{u}\right)^{\mu +(1/2)}
\frac{\Gamma \left(\frac12 \nu + \frac12 \mu +1\right)}
{\Gamma \left(\frac12 \nu - \frac12 \mu +\frac12 \right)}
\\
=1 - \frac{(4\mu^2-1)(2\mu + 3)}{48u^2}
+\mathcal{O}\left(\frac{1}{u^4}\right)
\quad  (u=\nu+\tfrac{1}{2} \to \infty).
\end{multline*}

Although points arbitrarily close to but not equal to $z=1$ are theoretically included in the domain of asymptotic validity, we expect the expansions of \cref{thm:nularge} to become less accurate for $z$ close to this singularity. For this reason we restrict use of the expansions to $|z-1| \geq r$ for some chosen $r \in (0,1)$ that is not too small.

For points inside this excluded disk we can instead use the Cauchy integral formula to obtain
\begin{multline}
\label{eq3.35}
A_{\mu}(u,z) 
=\frac{1}{2\pi i}\oint_{\mathscr{C}} 
\frac{A_{\mu}(u,t)}{t-z}dt
\\
\sim 
\frac{1}{2\pi i}\oint_{\mathscr{C}} 
\exp \left\{
\sum\limits_{s=1}^{\infty}
\frac{\tilde{\mathcal{E}}_{\mu,2s}(t)}{u^{2s}}
\right\}
\cosh \left\{ \sum\limits_{s=0}^{\infty }
\frac{\tilde{\mathcal{E}}_{\mu,2s+1}(t)}{u^{2s+1}}
\right\}
\frac{dt}{t-z}
\quad (|z-1| < r, \,u \to \infty),
\end{multline}
and similarly for $B_{\mu}(u,z)$. Here $\mathscr{C}$ is a suitably chosen positively orientated loop enclosing $t=1$ and $t=z$, for example the circle $|z-1|=r'$ where $r' \in (r,1]$. For each fixed $\mu$ values of a finite number of the coefficients $\tilde{\mathcal{E}}_{\mu,s}(t)$ and $\mathcal{E}_{\mu,s}(t)$ at a chosen discrete set of points on the contour can readily be computed and stored, with rapid numerical integration then used for varying values of $u$ and $z$ to approximate $A_{\mu}(u,z)$ and $B_{\mu}(u,z)$. See also \cref{remark:trad} below for an alternative method of asymptotically evaluating these coefficient functions near the pole (as well as those in \cref{thm.conical,thm:mularge} below).

In \cref{sec.numerics} we test our expansions in a similar manner for certain expressions involving our approximations, directly for $|z-1| \geq r$ (our choice is $r=\frac12$), and via a Cauchy integral for $|z-1| < r$, where we take the contour to be the circle $|z-1|=1$.

As stated in the theorem the expansions (\ref{eq3.24f}) and (\ref{eq3.24g}) are valid for $\mu \in \mathbb{C}$ (provided $\Re(\mu) \geq 0$), since the large parameter $u \in \mathbb{R}$ and the Liouville variable $\xi$ given by (\ref{eq2.6}) is unchanged, and hence so too are the associated regions of validity since they are determined by $\Re(u \xi)=u\Re(\xi)$ (see \cite[Chap. 10, sect. 3.1]{Olver:1997:ASF}).

The special case $\mu =i\rho$ where $\rho \in \mathbb{R}$ is of most interest when $\mu$ is not real. The expansions (\ref{eq3.24f}) and (\ref{eq3.24g}) still hold, simply with $\mu$ replaced by $i\rho$. Note that from (\ref{eq2.14}) - (\ref{eq2.17}), (\ref{eq2.31}), (\ref{eq2.32}), and (\ref{eq3.34}) the coefficients $\mathcal{E}_{i\rho,s}(z)$ are also real in the same circumstances. However each $a_{i\rho+1,s}$ is not real, and so from (\ref{eq3.34a}) neither is $\tilde{\mathcal{E}}_{i\rho,s}(z)$. Now from \cite[Eq. 14.9.14]{NIST:DLMF} $\boldsymbol{Q}^{-i\rho}_{\nu}(z)=\boldsymbol{Q}^{i\rho}_{\nu}(z)$, which incidentally means it is real for $\rho, \nu \in \mathbb{R}$ and $1 < z < \infty$. Also $K_{-\mu}(z)=K_{\mu}(z)$, and from (\ref{eq3.24L}) and (\ref{eq3.38}) $B_{-\mu}(u,z)=B_{\mu}(u,z)$. Thus if we set $\mu=\pm i \rho$ in (\ref{eq3.24ee}), add the two equations, then divide the sum by two and refer to \cite[Eq. 10.29.2]{NIST:DLMF}, we arrive at
\begin{multline}
\label{eq3.24ff}
\boldsymbol{Q}^{i\rho}_{\nu}(z)
=
\frac{\sqrt{(2\nu+1)\pi}\,\xi^{1/2}}{2^{\nu+1}
\left|\Gamma \left(\frac12 \nu+1 + \frac12 i\rho 
\right)\right|^2
\left(z^2-1\right)^{1/4}}
 \\ \times
\left\{K_{i\rho}(u \xi)\hat{A}_{i\rho}(u,z)
-\xi K_{i\rho}'(u \xi)B_{i\rho}(u,z)\right\},
\end{multline}
where
\begin{equation} 
\label{eq3.24gg}
\hat{A}_{\mu}(u,z)
=\tfrac12\left\{
A_{\mu}(u,z)+A_{-\mu}(u,z)
\right\}.
\end{equation}
The asymptotic expansion for $\hat{A}_{i\rho}(u,z)$ is readily obtained from (\ref{eq3.24f}). Note that $K_{i\rho}(u\xi)$, $K_{i\rho}'(u\xi)$, $B_{i\rho}(u,z)$ and $\hat{A}_{i\rho}(u,z)$ are all real when $\nu, \rho \in \mathbb{R}$ and $1 < z < \infty$. 

$P^{i\rho}_{\nu}(z)$ is not real in this case, but it remains the unique solution having the behaviour (\ref{eq1.1b}) as $z \to 1$. Now from (\ref{eq3.24e}) and (\ref{eq3.25})
\begin{multline}
\label{eq3.25a}
P_{\nu}^{ \pm i \rho}(z)
=L_{\nu}^{i \rho} \, \Gamma(\nu+1 \pm i \rho)
 \left(z^2-1\right)^{-1/4}
\\ \times
 \xi^{1/2} \left\{I_{\mp i \rho}(u \xi)A_{i \rho}(u,z)
+ \xi I_{\mp i \rho-1}(u \xi)B_{i \rho}(u,z)\right\}.
\end{multline}
These can, for example, be used to yield the asymptotic expansion for function $\hat{P}^{i\rho}_{\nu}(z)$, where $\hat{P}^{\mu}_{\nu}(z)$ is defined by
\begin{equation} 
\label{eq3.24hh}
\hat{P}^{\mu}_{\nu}(z)
=\tfrac12\left\{
P^{\mu}_{\nu}(z)
+P^{-\mu}_{\nu}(z)\right\}.
\end{equation}
This is real-valued when $\mu = i\rho$ ($\rho \in \mathbb{R}$) under the above conditions, and in this case from (\ref{eq1.1b}), (\ref{eq3.24hh}) and \cite[Eqs. 5.4.3 and 5.5.1]{NIST:DLMF} it has the oscillatory behaviour
\begin{equation*}
\label{eq3.24i}
\hat{P}^{i\rho}_{\nu}(z) \sim
\left(\frac{\sinh(\pi \rho)}{\pi \rho}\right)^{1/2}
\cos\left\{
\frac{\rho}{2} \ln\left(\frac{z-1}{2}\right)
-\phi(\rho)\right\}
\quad (z \rightarrow 1^{+}),
\end{equation*}
where $\phi(\rho)=\arg\{\Gamma(1+i\rho)\}$.

\section{Conical case: \texorpdfstring{$\nu=-\frac12+i\tau$}{} for large positive \texorpdfstring{$\tau$}{} and \texorpdfstring{$\mu$}{} bounded}
\label{sec.LargeTau}

Now consider $\nu=-\frac12+i \tau$ with $\tau \in \mathbb{R}$ and $\tau \to \infty$. Again $\mu$ will be bounded with $\Re(\mu) \geq 0$. In what follows the variables and coefficients are the same as given in \cref{thm:nularge}.

Although the Liouville transformation is unchanged we have real $u$ replaced by purely imaginary $i \tau$, and so the LG regions of validity need to be re-evaluated, since this time instead of $\Re(u\xi)=u\Re(\xi)$ being monotonic on the paths linking the reference point to the points where the LG expansions are valid, we require monotonicity for $\Re(i\tau\xi)=-\tau\Im(\xi)$. To this end, we first require the following LG expansions for the functions $\boldsymbol{Q}^{\mu}_{-\frac{1}{2}+i\tau}(z)$ and $\boldsymbol{Q}^{\mu}_{-\frac{1}{2}+i\tau,-1}(z)$.

\begin{lemma}
\label{lem.LGconical}

Let $N$ be a positive integer. Then for $\tau \in \mathbb{R}$ and $\tau \to \infty$, with  $\Re(\mu) \geq 0$ and $0 \leq |\mu| \leq \mu_{0}< \infty$
\begin{multline}
\label{eq2.49}
\boldsymbol{Q}^{\mu}_{-\frac{1}{2}+i\tau}(z) =
\frac{\sqrt{\tfrac{1}{2}\pi}}{\Gamma(1+i\tau)
\left(z^2-1\right)^{1/4}}
\\  \times
\exp \left\{ -i\tau\xi
+\sum\limits_{s=1}^{N-1} i^{s}
\frac{E_{\mu,s}(\beta)
-E_{\mu,s}(1)}{\tau^{s}}
\right\}
\left\{1+\mathcal{O}\left(\frac{1}{\tau^{N}}\right)\right\},
\end{multline}
and
\begin{multline}
\label{eq2.42}
\boldsymbol{Q}^{\mu}_{-\frac{1}{2}+i\tau,-1}(z) =
\frac{i\sqrt{\tfrac{1}{2}\pi}}{\Gamma(1+i\tau)
\left(z^2-1\right)^{1/4}}
\\  \times
\exp \left\{i\tau\xi
+\sum\limits_{s=1}^{N-1} (-i)^{s}
\frac{E_{\mu,s}(\beta)
-E_{\mu,s}(-1)}{\tau^{s}}
\right\}
\left\{1+\mathcal{O}\left(\frac{1}{\tau^{N}}\right)\right\},
\end{multline}
uniformly for $|\arg(z)|\leq\frac12 \pi$ and $|z-1| \geq \delta>0$.
\end{lemma}

\begin{remark}
The regions of validity are larger than stated, but we only require them for the (principal) right half $z$ plane. For example, (\ref{eq2.42}) holds on the negative imaginary $z$ axis on the sheet traversing the cut $[0,1]$ from above, which corresponds to the line $\mathsf{A}$$\mathsf{D}_{+1}'$ in the $\xi$ plane as shown in \cref{fig:Figxiplane}. Moreover, the order term  vanishes as $z \to \infty_{+1}$ (the point at infinity on the negative imaginary axis across the cut and corresponding to $\xi \to -\infty + \frac12 \pi i$). The order term in (\ref{eq2.49}) vanishes as $z \to \infty$ in the half plane $|\arg(z)|\leq\frac12 \pi$ ($\Re(\xi) \to +\infty$).
\end{remark}

\begin{proof}
The expansions on the RHS of (\ref{eq2.49}) and (\ref{eq2.42}) come from \cite[Sect. 1]{Dunster:2020:LGE} applied to (\ref{eq2.3}), with $u=i\tau$. Bounds for the error terms are furnished by \cite[Thm. 1.1]{Dunster:2020:LGE}, and these, along with the domains of validity, depend on suitably chosen reference points $\xi = \alpha_{1,2}$.

For (\ref{eq2.49}) we take the reference point $\xi=\alpha_{2}=\infty-\frac12 \pi i$, i.e. the point at infinity labelled $\mathsf{D}'$ in \cref{fig:Figxiplane}. Then the region of validity consists of points $\xi$ that can be linked to $\alpha_{2}$ by a finite chain of $R_{2}$ arcs (as defined in \cite[Chap. 5, \S 3.3]{Olver:1997:ASF}) with the property that as $t$ (say) passes along the path from $\alpha_{2}$ to $\xi$, the imaginary part of $t$ is nondecreasing. In addition the point $\xi=0$ (corresponding to $z=1$) must be excluded, as the error terms diverge at this singularity. Now $|\arg(z)|\leq \frac12 \pi$ corresponds to  the strip $0 \leq \Re(\xi) < \infty$, $|\Im(\xi)| \leq \frac12 \pi$. From \cref{fig:Figxiplane} we see that all points $\xi$ in this strip are accessible by a straight line path to $\infty-\frac12 \pi i$, and along this path the monotonicity requirement clearly holds. This verifies that the LG asymptotic expansion of the RHS of (\ref{eq2.49}) holds in the stated region. 

For (\ref{eq2.42}) we take $\alpha_{1}=-\infty + \frac12 \pi i$ which is labelled by $\mathsf{D}_{+1}'$ in \cref{fig:Figxiplane}. This time the monotonicity requirement is that the imaginary part of $t$ is nonincreasing as $t$ passes along the path from $\alpha_{1}$ to $\xi$. From \cref{fig:Figxiplane} we see that all points in the strip $0 \leq \Re(\xi) < \infty$, $|\Im(\xi)| \leq \frac12 \pi$ are accessible by a path to $\alpha_{1}$ along which this monotonicity requirement holds. For example, the path could consist of the union of the horizontal line from $\alpha_{1}$ to $1 + \frac12 \pi i$, and a straight line from this latter point to $\xi$. This again verifies that the LG asymptotic expansion of the RHS of (\ref{eq2.42}) holds in the stated region. 

Finally the matching of both sides of these two equations follows from the uniqueness of the oscillatory behaviour of the solutions at $z=\infty$ ($\xi =+\infty$) and $z=\infty_{+1}$ respectively, recalling that the latter denotes the point at infinity corresponding to $\xi = \alpha_{1}$, i.e. the point at infinity on the imaginary $z$ axis after crossing the cut $[0,1]$ from above. Now from (\ref{eq2.6})
\begin{equation}
\label{eq2.6b}
\xi = -\ln(2|z|)+\mathcal{O}(z^{-2})
\quad (z \to \infty_{+1}).
\end{equation}
The constants of proportionality in both expansions follow from (\ref{eq2.22y}), (\ref{eq1.1c}), (\ref{eq2.6}) and (\ref{eq2.6b}). In this we used that $\beta \to \pm 1$ as  $z \to \infty$ and $z \to \infty_{+1}$, respectively.
\end{proof}

For conical functions we now state our main result.

\begin{theorem}
\label{thm.conical}
\begin{multline}
\label{eq9.1b}
P^{-\mu}_{-\frac{1}{2}+i\tau}(z)
=e^{\mu \pi i/2} L_{-\frac{1}{2}+i\tau}^{\mu} \, 
\Gamma\left(\tfrac12-\mu+i\tau\right)
\left(z^2-1\right)^{-1/4} 
\\ \times
\xi^{1/2}\left\{J_{\mu}(\tau \xi)
A_{\mu}(i\tau,z)
+i \xi J_{\mu+1}(\tau \xi )
B_{\mu}(i\tau,z)\right\},
\end{multline}
\begin{multline}
\label{eq9.1c}
\boldsymbol{Q}^{\mu}_{-\frac{1}{2}+i\tau}(z)
=-\tfrac12 \pi i e^{-\mu \pi i/2}
L_{-\frac{1}{2}+i\tau}^{\mu} \, 
\left(z^2-1\right)^{-1/4} 
\\ \times
\xi^{1/2}\left\{H_{\mu}^{(2)}(\tau \xi)
A_{\mu}(i\tau,z)
+i \xi H_{\mu+1}^{(2)}(\tau \xi )
B_{\mu}(i\tau,z)\right\},
\end{multline}
\begin{multline}
\label{eq9.1d}
P^{\mu}_{-\frac{1}{2}+i\tau}(z)
=e^{-\mu \pi i/2} L_{-\frac{1}{2}+i\tau}^{\mu} \, 
\Gamma\left(\tfrac12+\mu+i\tau\right)
\left(z^2-1\right)^{-1/4} 
\\ \times
\xi^{1/2}\left\{J_{-\mu}(\tau \xi)
A_{\mu}(i\tau,z)
+i \xi J_{-\mu-1}(\tau \xi )
B_{\mu}(i\tau,z)\right\},
\end{multline}
\begin{multline}
\label{eq9.1f}
\boldsymbol{Q}^{\mu}_{-\frac{1}{2}+i\tau,+1}(z)
=\tfrac12 \pi i e^{\mu \pi i/2}
L_{-\frac{1}{2}+i\tau}^{\mu} \, 
\left(z^2-1\right)^{-1/4} 
\\ \times
\xi^{1/2}\left\{H_{\mu}^{(1)}\left(\tau \xi e^{2\pi i}\right)
A_{\mu}(i\tau,z)
+ i \xi H_{\mu+1}^{(1)}\left(\tau \xi e^{2\pi i}\right)
B_{\mu}(i\tau,z)\right\},
\end{multline}
and
\begin{multline}
\label{eq9.1e}
\boldsymbol{Q}^{\mu}_{-\frac{1}{2}+i\tau,-1}(z)
=\tfrac12 \pi i e^{\mu \pi i/2}
L_{-\frac{1}{2}+i\tau}^{\mu} \, 
\left(z^2-1\right)^{-1/4} 
\\ \times
\xi^{1/2}\left\{H_{\mu}^{(1)}(\tau \xi)
A_{\mu}(i\tau,z)
+ i \xi H_{\mu+1}^{(1)}(\tau \xi)
B_{\mu}(i\tau,z)\right\},
\end{multline}
where $A_{\mu}(i\tau,z)$ and $B_{\mu}(i\tau,z)$ are analytic for $|\arg(z)| \leq \frac12 \pi$, and for $\mu$ bounded with $\Re(\mu) \geq 0$ possess the asymptotic expansions
\begin{equation}
\label{eq3.24fff}
A_{\mu}(i\tau,z) \sim 
\exp \left\{
\sum\limits_{s=1}^{\infty}
(-1)^{s}\frac{\tilde{\mathcal{E}}_{\mu,2s}(z)}{\tau^{2s}}
\right\}
\cos \left\{ \sum\limits_{s=0}^{\infty } (-1)^{s}
\frac{\tilde{\mathcal{E}}_{\mu,2s+1}(z)}{\tau^{2s+1}}
\right\},
\end{equation}
and
\begin{equation}
\label{eq3.24ggg}
B_{\mu}(i\tau,z) \sim -\frac{i}{\xi}
\exp \left\{
\sum\limits_{s=1}^{\infty}
(-1)^{s}\frac{\mathcal{E}_{\mu,2s}(z)}{\tau^{2s}}
\right\}
\sin \left\{ \sum\limits_{s=0}^{\infty } (-1)^{s}
\frac{\mathcal{E}_{\mu,2s+1}(z)}{\tau^{2s+1}}
\right\},
\end{equation}
as $\tau \to \infty$ uniformly for  $|\arg(z)| \leq \frac12 \pi$ with $|z-1| \geq \delta >0$.
\end{theorem}

\begin{remark}
From (\ref{eq3.24L}) we obtain
\begin{equation}
\label{eq3.24j}
L_{-\frac{1}{2}+i\tau}^{\mu}=
\frac{e^{\pi i/4} \sqrt{\pi\tau}}
{2^{i\tau}\Gamma \left(
\frac12 \mu +\frac34 +\frac12 i\tau  \right)
\Gamma \left(
\frac34- \frac12 \mu +\frac12 i\tau  \right)},
\end{equation}
and hence on referring to the gamma functional relations \cite[Sect. 5.5]{NIST:DLMF} it follows that
\begin{multline*}
\label{eq3.24k}
e^{\pm \mu \pi i/2} L_{-\frac{1}{2}+i\tau}^{\mu} \, 
\Gamma\left(\tfrac12 \mp \mu +i\tau \right)
=2^{\mp \mu-(3/2)} \pi^{-1} \sqrt{\tau}\, e^{\pi \tau/2}
\\ \times
\left|\Gamma \left( \tfrac14
\mp \tfrac12 \mu +\tfrac12 i\tau\right)\right|^{2}
\left(1+i e^{\pm \mu \pi i}e^{-\pi \tau}\right),
\end{multline*}
which are therefore real to within a relatively exponentially small imaginary term for large $\tau$. Thus the expansions given by (\ref{eq9.1b}), (\ref{eq9.1d}), (\ref{eq3.24fff}) and (\ref{eq3.24ggg}) are consistent with the known property that $P^{\mu}_{-\frac{1}{2}+i\tau}(x)$ is real for $\tau,\mu \in \mathbb{R}$ and $1<x<\infty$.
\end{remark}

\begin{remark}
Expansions for $|z-1|<\delta$ can be achieved via Cauchy's integral formula similarly to (\ref{eq3.35}).
\end{remark}

\begin{proof}
Firstly, (\ref{eq9.1b}) - (\ref{eq9.1e}) follow from (\ref{eq3.24e}) - (\ref{eq9.1a}) and using \cite[Eqs. 10.27.6 and 10.27.8]{NIST:DLMF}. Analyticity of the coefficients $A_{\mu}(i\tau,z)$ and $B_{\mu}(i\tau,z)$ at $z=1$ can readily be established in a similar manner to the case $\nu \in \mathbb{R}$.

It remains to verify that (\ref{eq3.24f}) and (\ref{eq3.24g}) hold with real $u$ replaced by purely imaginary $i\tau$, which formally (but not rigorously) yield the expansions (\ref{eq3.24fff}) and (\ref{eq3.24ggg}). To establish that (\ref{eq3.24fff}) indeed holds we use (\ref{eq9.1c}), (\ref{eq9.1e}) and \cite[Eq. 10.5.5]{NIST:DLMF} to obtain
\begin{multline} 
\label{eq2.27a}
A_{\mu}(i\tau,z) = 
-\frac{\tau  \left(z^2-1\right)^{1/4} \xi^{1/2}}
{2 L_{-\frac{1}{2}+i\tau}^{\mu}}
\left\{ e^{\mu \pi i/2}
\boldsymbol{Q}^{\mu}_{-\frac{1}{2}+i\tau}(z)
H_{\mu+1}^{(1)}(\tau \xi)
\right.
\\
\left.
+ e^{-\mu \pi i/2}
\boldsymbol{Q}^{\mu}_{-\frac{1}{2}+i\tau,-1}(z)
H_{\mu+1}^{(2)}(\tau \xi)
\right\}.
\end{multline}
Next from (\ref{eq2.30}) and \cite[Eq. 10.27.8]{NIST:DLMF}
\begin{multline} 
\label{eq2.30a}
H_{\mu+1}^{(1)}(\tau \xi) \sim
e^{-(2\mu+3)\pi i/4}
\left(\frac{2}{\pi \tau \xi }\right)^{1/2}
\exp\left\{i \tau \xi
-\sum_{s=1}^{\infty} (-i)^s
\frac{a_{\mu+1,s}}{s\, (\tau \xi)^{s}}
\right\}
\\
\quad   \left(\tau \xi \rightarrow \infty, \,
-\pi + \delta \leq \arg(\xi) \leq 2\pi - \delta   \right),
\end{multline}
and
\begin{multline} 
\label{eq2.30b}
H_{\mu+1}^{(2)}(\tau \xi) \sim
e^{(2\mu+3)\pi i/4}
\left(\frac{2}{\pi \tau \xi }\right)^{1/2}
\exp\left\{-i \tau \xi
-\sum_{s=1}^{\infty} i^s
\frac{a_{\mu+1,s}}{s\, (\tau \xi)^{s}}
\right\}
\\
\quad   \left(\tau \xi \rightarrow \infty, \,
-2\pi + \delta \leq \arg(\xi) \leq \pi - \delta \right).
\end{multline}
Now (\ref{eq2.49}), (\ref{eq2.42}), (\ref{eq2.30a}) and (\ref{eq2.30b}) are all valid for $\xi$ lying in the domain $0\leq \Re(\xi)<\infty$, $|\Im(\xi)| \leq \frac12 \pi$, $|\xi| \geq \delta >0$, or equivalently for $z$ lying in the stated region; see \cref{fig:Figzplane,fig:Figxiplane}. Then from (\ref{eq2.27b}), (\ref{eq2.49}), (\ref{eq2.42}) and (\ref{eq2.27a}) - (\ref{eq2.30b}) we verify that (\ref{eq3.24fff}) holds as $\tau \to \infty$ uniformly in this region. The expansion (\ref{eq3.24ggg}) in the same domain can be established similarly.
\end{proof}

As earlier there is again no restriction on $\mu$ being real. In particular, expansions for $P^{\mp i\rho}_{-\frac{1}{2}+i\tau}(z)$ and $\boldsymbol{Q}^{i\rho}_{-\frac{1}{2}+i\tau}(z)$ ($\rho >0$) come immediately from (\ref{eq9.1b}) - (\ref{eq9.1d}), (\ref{eq3.24fff}) and (\ref{eq3.24ggg}) with $\mu$ replaced by $i\rho$.

\section{Large \texorpdfstring{$|\mu|$}{} with \texorpdfstring{$\nu$}{} bounded}
\label{sec.LargeMu}

As before $u$, $\beta$ and $L_{\mu}^{\nu}$ are given by (\ref{eq2.2}), (\ref{eq2.12}) and (\ref{eq3.24L}), respectively.

\begin{theorem}
\label{thm:mularge}
Let 
\begin{equation}
\label{eq4.4a}
\hat{\xi}=\mathrm{arccosh}(\beta)
=\mathrm{arccoth}(z).
\end{equation}
Then
\begin{equation}
\label{eq4.3}
P_{\nu}^{-\mu}(z)
=\sqrt{2/\pi}\, L_{\mu-\frac12}^{u} 
\hat{\xi}^{1/2}
\left\{K_{u}(\mu \hat{\xi})A_{u}(\mu,\beta)
-\hat{\xi} K_{u+1}(\mu \hat{\xi})B_{u}(\mu,\beta)\right\},
\end{equation}
and
\begin{equation}
\label{eq4.4}
\boldsymbol{Q}^{\mu}_{\nu}(z)
=\sqrt{\tfrac{1}{2}\pi} \, L_{\mu-\frac12}^{u}
\Gamma(\mu-\nu)
 \hat{\xi}^{1/2} \left\{I_{u}(\mu \hat{\xi})A_{u}(\mu,\beta)
+\hat{\xi} I_{u+1}(\mu \hat{\xi})B_{u}(\mu,\beta)\right\},
\end{equation}
where $A_{u}(\mu,\beta)$ and $B_{u}(\mu,\beta)$ are analytic for $|\arg(z)| \leq \frac12 \pi$. Moreover, as $\mu \rightarrow \infty$, uniformly for $|\arg(z)| \leq \frac12 \pi$ and $|\beta-1| \geq \delta>0$, and $\Re(u) \geq 0$, $|u| \leq u_{0} < \infty$, they possess the asymptotic expansions
\begin{equation}
\label{eq4.5}
A_{u}(\mu,\beta) \sim 
\exp \left\{
\sum\limits_{s=1}^{\infty}
\frac{\tilde{\mathcal{E}}_{u,2s}(\beta)}{\mu^{2s}}
\right\}
\cosh \left\{ \sum\limits_{s=0}^{\infty }
\frac{\tilde{\mathcal{E}}_{u,2s+1}(\beta)}{\mu^{2s+1}}
\right\},
\end{equation}
and
\begin{equation}
\label{eq4.6}
B_{u}(\mu,\beta) \sim 
\frac{1}{\hat{\xi}}\exp \left\{
\sum\limits_{s=1}^{\infty}
\frac{\mathcal{E}_{u,2s}(\beta)}{\mu^{2s}}
\right\}
\sinh \left\{ \sum\limits_{s=0}^{\infty }
\frac{\mathcal{E}_{u,2s+1}(\beta)}{\mu^{2s+1}}
\right\},
\end{equation}
where for $s=1,2,3,\ldots$
\begin{equation}
\label{eq4.7}
\mathcal{E}_{u,s}(\beta)
=E_{u,s}(z)+(-1)^{s+1}\frac{a_{u,s}}
{s{\hat{\xi}}^{s}},
\end{equation}
and
\begin{equation}
\label{eq4.8}
\tilde{\mathcal{E}}_{u,s}(\beta)
=E_{u,s}(z)+(-1)^{s+1}\frac{a_{u+1,s}}
{s\hat{\xi}^{s}}.
\end{equation}
Here $E_{u,s}(z)$ are polynomials in $z$, given by (\ref{eq2.14}) - (\ref{eq2.17}) with $\beta$ and $\mu$ replaced by $z$ and $u$ respectively, and $a_{u,s}$ are given by (\ref{eq2.31}) and (\ref{eq2.32}) with $\mu$ replaced by $u$.
\end{theorem}

\begin{proof}
We use the Whipple-type formulas \cite[Eqs. 14.9.11, 14.9.16, 14.9.14 and 14.9.17]{NIST:DLMF}
\begin{equation}
\label{eq4.1}
P^{-\mu}_{\nu}(z)
=\sqrt{2/\pi}\left(z^{2}-1\right)^{-1/4}
\boldsymbol{Q}^{u}
_{\mu-\frac12}(\beta),
\end{equation}
and
\begin{equation}
\label{eq4.2}
\boldsymbol{Q}^{\mu}_{\nu}(z)
=\sqrt{\tfrac{1}{2}\pi}
\left(z^{2}-1\right)^{-1/4} 
P^{-u}_{\mu-\frac12}(\beta),
\end{equation}
where $u$ and $\beta$ are given by (\ref{eq2.2}) and (\ref{eq2.12}), respectively.

Now the half plane $|\arg(z)| \leq \frac12 \pi$ with a cut along $[0,1]$ and principal square root taken is mapped to the same region in the $\beta$ plane: see \cite[Figs. 4 and 5]{Dunster:2024:AEG} which illustrate the mapping of the first quadrant in the $z$ plane to the fourth quadrant of the $\beta$ plane; the fourth quadrant in the $z$ plane is similarly mapped to the first quadrant of the $\beta$ plane by taking conjugates of both these regions.

Thus from (\ref{eq4.1}) and (\ref{eq4.2}) we can directly apply (\ref{eq3.24e}) and (\ref{eq3.24ee}), with $z \mapsto \beta$, $\mu \mapsto u$ and $\nu \mapsto \mu - \frac12$ (equivalently $u \mapsto \mu$), to obtain (\ref{eq4.3}) and (\ref{eq4.4}). Similarly under these transformations the expansions (\ref{eq4.5}) - (\ref{eq4.8}) follow from (\ref{eq3.24f}) - (\ref{eq3.34a}), on noting from (\ref{eq2.12}) that $\beta = \beta(z)$ is an involution mapping, i.e. 
\begin{equation}
\label{eq4.8a}
z=\beta (\beta^2-1)^{-1/2}.
\end{equation}
\end{proof}

This time the expansions (\ref{eq4.5}) and (\ref{eq4.6}) as they stand are only practicable for bounded $|z|$, since the coefficients become unbounded as $z \to \infty$ ($\beta \to 1$). In this case we have from (\ref{eq2.12}) that $\beta$ is bounded, and so instead for large $z$ we use Cauchy's integral formula to approximate $A_{u}(\mu,\beta)$ and $B_{u}(\mu,\beta)$, in a similar manner for small $|z-1|$ in the previous section. 

To show this, let
\begin{equation}
\label{eq4.8b}
A_{u}(N,\mu,\beta) =
\exp \left\{
\sum\limits_{s=1}^{N}
\frac{\tilde{\mathcal{E}}_{u,2s}(\beta)}{\mu^{2s}}
\right\}
\cosh \left\{ \sum\limits_{s=0}^{N}
\frac{\tilde{\mathcal{E}}_{u,2s+1}(\beta)}{\mu^{2s+1}}
\right\},
\end{equation}
where $N$ is any positive integer. In this we require that $\tilde{\mathcal{E}}_{u,s}(\beta)$ ($s=1,2,3,\ldots 2N+1$), rather than polynomials in $z$, be expressed in terms of $\beta$ (see (\ref{eq4.4a}) and (\ref{eq4.8})). Thus from (\ref{eq4.8a}), for each $E_{u,s}(z)$ we replace $\beta$ by $\beta (\beta^2-1)^{-1/2}$ in $E_{\mu,s}(\beta)$ (as well as $\mu$ by $u$), instead of replacing $\beta$ by $z$. For example, from (\ref{eq2.17a}), (\ref{eq2.31}), (\ref{eq4.4a}) and (\ref{eq4.8}) we express the first coefficient in (\ref{eq4.8b}) in the form
\begin{equation*}
\label{eq2.17aa}
\tilde{\mathcal{E}}_{u,1}(\beta)
=-\frac{\left(4u^2-1\right) \beta}
{8(\beta^2-1)^{1/2}}
+\frac{4(u+1)^2-1}
{8\, \mathrm{arccosh}(\beta)}.
\end{equation*}
Then since $A_{u}(\mu,\beta)$ is analytic in the half-plane $|\arg(\beta)| \leq \frac12 \pi$ (including $\beta =1$) we have
\begin{equation}
\label{eq4.8c}
A_{u}(\mu,\beta) 
=\frac{1}{2\pi i}\oint_{\mathscr{C}} 
\frac{A_{u}(\mu,b)}{b-\beta}db
\approx
\frac{1}{2\pi i}\oint_{\mathscr{C}} 
\frac{A_{u}(N,\mu,b)}{b-\beta}db
\quad (\mu \to \infty),
\end{equation}
where $\mathscr{C}$ is a suitably chosen positively orientated closed loop in the half plane $|\arg(b)| \leq \frac12 \pi$ enclosing $b=\beta$ and $b=1$, with both these points not too close to the contour. 

Note that $\beta$ is close to $1$ for large $z$, which is why we use the contour integral, since points on the contour are bounded away from $\beta=1$ (where the coefficients in (\ref{eq4.5}) and (\ref{eq4.6}) are unbounded). Using a Cauchy integral $B_{u}(\mu,\beta)$ can similarly be approximated for large $z$ ($\beta$ close to 1).

Again as stated in \cref{thm:mularge} the bounded parameter, in this case $\nu$, does not have to be real, whether we use the expansions (\ref{eq4.5}) and (\ref{eq4.6}) directly or via Cauchy's integral formula. In the conical case $\nu = -\frac12 +i \tau$ ($u=i \tau$) we have here that $\tau >0$ is bounded. So for example, from (\ref{eq4.3}),
\begin{equation*}
\label{eq4.9}
P_{-\frac12 +i \tau}^{-\mu}(z)
=\sqrt{2/\pi}\, L_{\mu-\frac12}^{i \tau} 
\hat{\xi}^{1/2}
\left\{K_{i \tau}(\mu \hat{\xi})A_{i \tau}(\mu,\beta)
-\hat{\xi} K_{1 +i \tau}(\mu \hat{\xi})B_{i \tau}(\mu,\beta)\right\}.
\end{equation*}
If we then replace $\tau \mapsto -\tau$, add it to the original, divide by 2, refer to \cite[Eq. 10.29.1]{NIST:DLMF}, and note that $P_{-\frac12 +i \tau}^{-\mu}(z)$, $L_{\mu-\frac12}^{i \tau}$, $K_{i \tau}(\mu \hat{\xi})$ and $B_{i \tau}(\mu,\beta)$ are all even in $\tau$, we arrive at the more convenient representation
\begin{equation}
\label{eq4.10}
P_{-\frac12 +i \tau}^{-\mu}(z)
=\sqrt{2/\pi}\, L_{\mu-\frac12}^{i \tau} 
\hat{\xi}^{1/2}
\left\{K_{i \tau}(\mu \hat{\xi})\hat{A}_{i \tau}(\mu,\beta)
-\hat{\xi} K_{i \tau}^{\prime}(\mu \hat{\xi})
B_{i \tau}(\mu,\beta)\right\},
\end{equation}
where $\hat{A}_{i \tau}(\mu,\beta)$ is given by (\ref{eq3.24gg}) with $\mu$, $u$ and $z$ replaced by $i \tau$, $\mu$ and $\beta$ respectively. The expansions (\ref{eq4.5}) - (\ref{eq4.8}) then of course remain valid with $u=\pm i \tau$. An expansion for $\boldsymbol{Q}^{\mu}_{-\frac12 \pm i \tau}(z)$ comes directly from (\ref{eq4.4}) - (\ref{eq4.8}) on setting $\nu =-\frac12 \pm i \tau$ ($u =\pm i \tau$).

For $\mu=i\rho$ with $\rho \to \infty$ from \cref{thm.conical}, (\ref{eq4.1}) and (\ref{eq4.2}) we obtain the following. 

\begin{theorem}
\begin{multline}
\label{eq4.11}
P^{-i\rho}_{\nu}(z)
=-i e^{-u \pi i/2} \sqrt{\tfrac{1}{2}\pi}\,
L_{-\frac{1}{2}+i\rho}^{u} 
\\ \times
\hat{\xi}^{1/2}\left\{H_{u}^{(2)}(\rho \hat{\xi})
A_{u}(i\rho,\beta)
+i \hat{\xi} H_{u+1}^{(2)}(\rho \hat{\xi} )
B_{u}(i\rho,\beta)\right\},
\end{multline}
and
\begin{multline}
\label{eq4.12}
\boldsymbol{Q}^{i\rho}_{\nu}(z)
=e^{u \pi i/2} \sqrt{\tfrac{1}{2}\pi}\,
 L_{-\frac{1}{2}+i\rho}^{u} \, 
\Gamma\left(\tfrac12-u+i\rho\right)
\\ \times
\hat{\xi}^{1/2}\left\{J_{u}(\rho \hat{\xi})
A_{u}(i\rho,\beta)
+i \hat{\xi} J_{u+1}(\rho \hat{\xi} )
B_{u}(i\rho,\beta)\right\},
\end{multline}
where $A_{u}(i\rho,\beta)$ and $B_{u}(i\rho,\beta)$ are analytic for $|\arg(z)| \leq \frac12 \pi$, and for $\Re(u) \geq 0$, $|u| \leq u_{0} < \infty$ possess the asymptotic expansions
\begin{equation}
\label{eq4.13}
A_{u}(i\rho,\beta) \sim 
\exp \left\{
\sum\limits_{s=1}^{\infty}
(-1)^{s}\frac{\tilde{\mathcal{E}}_{u,2s}(\beta)}{\rho^{2s}}
\right\}
\cos \left\{ \sum\limits_{s=0}^{\infty } (-1)^{s}
\frac{\tilde{\mathcal{E}}_{u,2s+1}(\beta)}{\rho^{2s+1}}
\right\},
\end{equation}
and
\begin{equation}
\label{eq4.14}
B_{u}(i\rho,\beta) \sim -\frac{i}{\xi}
\exp \left\{
\sum\limits_{s=1}^{\infty}
(-1)^{s}\frac{\mathcal{E}_{u,2s}(\beta)}{\rho^{2s}}
\right\}
\sin \left\{ \sum\limits_{s=0}^{\infty } (-1)^{s}
\frac{\mathcal{E}_{u,2s+1}(\beta)}{\rho^{2s+1}}
\right\},
\end{equation}
as $\rho \to \infty$, uniformly for  $|\arg(z)| \leq \frac12 \pi$ with $|\beta-1| \geq \delta >0$.
\end{theorem}

Note these expansions apply in the conical case too, by simply replacing $u$ by $i\tau$ ($\tau>0$ bounded).

\section{Expansions for Ferrers functions}
\label{sec.Ferrers}

\subsection{Large \texorpdfstring{$|\nu|$}{} with \texorpdfstring{$\mu$}{} bounded}

As $x \rightarrow 1^{-}$ we have from \cite[Eqs. 14.8.1 and 14.8.6]{NIST:DLMF}
\begin{equation}
\label{eq5.2a}
\mathsf{P}_{\nu}^{-\mu}(x)
=\frac{1}{\Gamma(\mu+1)}
\left(\frac{1-x}{2}\right)^{\mu/2}
\left\{1+\mathcal{O}(1-x)\right\}
\quad (\mu \neq -1,-2,-3, \cdots),
\end{equation}
and
\begin{multline}
\label{eq5.2b}
\mathsf{Q}_{\nu}^{-\mu}(x)
=\frac{\Gamma(\mu)\Gamma(\nu-\mu+1)}
{2\Gamma(\nu+\mu+1)}
\left(\frac{2}{1-x}\right)^{\mu/2}
\left\{1+\mathcal{O}(1-x)\right\}
\\ 
\quad (\nu \pm \mu \neq -1,-2,-3, \cdots).
\end{multline}

The important property of $\mathsf{P}_{\nu}^{-\mu}(x)$ is that it is recessive at $x=1$, whereas $\mathsf{Q}_{\nu}^{-\mu}(x)$ is dominant under the above parameter conditions. Note also from \cite[Eqs. 14.2.6 and 14.9.1]{NIST:DLMF}
\begin{equation*}
\label{eq5.16b}
\mathscr{W}\left\{\mathsf{P}_{\nu}^{-\mu}(x),
\mathsf{Q}_{\nu}^{-\mu}(x)\right\}
=\frac{\Gamma(\nu-\mu+1)}{\Gamma(\nu+\mu+1)
\left(1-x^2\right)}.
\end{equation*}
Thus from (\ref{eq5.2a}) and (\ref{eq5.2b}) they form a numerically satisfactory pair in $[0,1)$ as $|\nu| \to \infty$ with $\mu$ bounded and non-negative. In this case we have the following for $x$ in this interval, with extension to $x \in (-1,0)$ achieved from the connection formulas \cite[Eqs. 14.9.8 and 14.9.10]{NIST:DLMF}.

\begin{theorem}
Let $L_{\nu}^{\mu}$ be given by (\ref{eq3.24L}) and for $x \in [0,1)$
\begin{equation}
\label{eq5.4}
\eta=\mathrm{arccos}(x) \in [0,1),
\quad
\gamma=x \left(1-x^2\right)^{-1/2} \in (0,\infty).
\end{equation}
Then
\begin{multline}
\label{eq5.9}
\mathsf{P}_{\nu}^{-\mu}(x)
=L_{\nu}^{\mu} \, \Gamma(\nu-\mu+1)
 \left(1-x^2\right)^{-1/4}
\\ \times
 \eta^{1/2} \left\{J_{\mu}(u \eta)A_{\mu}(u,x)
-\eta J_{\mu+1}(u \eta)B_{\mu}(u,x)\right\},
\end{multline}
and 
\begin{multline}
\label{eq5.10}
\mathsf{Q}_{\nu}^{-\mu}(x)
=-\tfrac12 \pi L_{\nu}^{\mu} \, \Gamma(\nu-\mu+1)
 \left(1-x^2\right)^{-1/4}
\\ \times
 \eta^{1/2} \left\{Y_{\mu}(u \eta)A_{\mu}(u,x)
-\eta Y_{\mu+1}(u \eta)B_{\mu}(u,x)\right\},
\end{multline}
where
\begin{equation}
\label{eq5.5}
A_{\mu}(u,x) \sim 
\exp \left\{
\sum\limits_{s=1}^{\infty}
\frac{\tilde{\mathcal{F}}_{\mu,2s}(x)}{u^{2s}}
\right\}
\cos \left\{ \sum\limits_{s=0}^{\infty }
\frac{i\tilde{\mathcal{F}}_{\mu,2s+1}(x)}{u^{2s+1}}
\right\},
\end{equation}
and
\begin{equation}
\label{eq5.6}
B_{\mu}(u,x) \sim 
\frac{1}{\eta}\exp \left\{
\sum\limits_{s=1}^{\infty}
\frac{\mathcal{F}_{\mu,2s}(x)}{u^{2s}}
\right\}
\sin \left\{\sum\limits_{s=0}^{\infty }
\frac{i\mathcal{F}_{\mu,2s+1}(x)}{u^{2s+1}}
\right\},
\end{equation}
as $u=\nu +\frac12 \rightarrow \infty$, uniformly for $0 \leq x \leq 1-\delta<1$ and $0 \leq \mu \leq \mu_{0}< \infty $. Here for $s=1,2,3,\ldots$
\begin{equation}
\label{eq5.7}
\mathcal{F}_{\mu,s}(x)
=E_{\mu,s}(i\gamma)-(-i)^{s}\frac{a_{\mu,s}}
{s\eta^{s}},
\end{equation}
and
\begin{equation}
\label{eq5.8}
\tilde{\mathcal{F}}_{\mu,s}(x)
=E_{\mu,s}(i\gamma)-(-i)^{s}\frac{a_{\mu+1,s}}
{s\eta^{s}}.
\end{equation}
\end{theorem}

\begin{remark}
Since $E_{\mu,2j}(\beta)$ and $E_{\mu,2j-1}(\beta)$ ($j=1,2,3,\ldots$) are even and odd functions, respectively, it follows that $\mathcal{F}_{\mu,2j}(x)$ and $i\mathcal{F}_{\mu,2j-1}(x)$ are real. For example from (\ref{eq2.17a}), (\ref{eq2.31}), (\ref{eq5.7}) and (\ref{eq5.8})
\begin{equation*}
\label{eq5.7a}
i\mathcal{F}_{\mu,1}(x)
=\frac18 \left(4\mu^2-1\right)
\left(\gamma-\frac{1}{\eta}\right), \;
\mathcal{F}_{\mu,2}(x)
=\frac{1}{16} \left(4\mu^2-1\right)
\left(\gamma^2 + \frac{1}{\eta^2}\right),
\end{equation*}
and
\begin{multline*}
\label{eq5.8a}
i\tilde{\mathcal{F}}_{\mu,1}(x)
=\frac18 
\left\{ \left(4\mu^2-1\right) \gamma
-\frac{4(\mu+1)^2-1}{\eta}\right\},
\\
\tilde{\mathcal{F}}_{\mu,2}(x)
=\frac{1}{16}
\left\{ \left(4\mu^2-1\right) \gamma^{2}
+\frac{4(\mu+1)^2-1}{\eta^{2}}\right\},
\end{multline*}
where $\eta=\eta(x)$ and $\gamma=\gamma(x)$ are given by (\ref{eq5.4}).
\end{remark}

\begin{remark}
\label{remark:trad}
To asymptotically approximate the coefficients near $x=1$ we can formally expand (\ref{eq5.5}) and (\ref{eq5.6}) in traditional asymptotic forms involving inverse powers of $u$, namely
\begin{equation}
\label{eq6.29}
A_{\mu}(u,x) \sim 
1+\sum\limits_{s=1}^{\infty}
\frac{\mathrm{A}_{\mu,2s}(x)}{u^{2s}},
\end{equation}
and
\begin{equation}
\label{eq6.30}
B_{\mu}(u,x) \sim 
\sum\limits_{s=0}^{\infty}
\frac{\mathrm{B}_{\mu,2s+1}(x)}{u^{2s+1}}.
\end{equation}

Both sets of coefficients are readily obtained explicitly, and are polynomials in $\gamma$ and $\eta^{-1}$. By virtue of \cite[Thm. 3.1]{Dunster:2021:NKF} and temporarily regarding $x \in \mathbb{C}$, one finds that each coefficient $\mathrm{A}_{\mu,2s}(x)$ and $\mathrm{B}_{\mu,2s+1}(x)$ is bounded, and in fact has a removable singularity, at $x=1$, even though the coefficients $\tilde{\mathcal{F}}_{\mu,s}(x)$ and $\mathcal{F}_{\mu,s}(x)$ in (\ref{eq5.5}) and (\ref{eq5.6}) are unbounded there. Thus if we define $\mathrm{A}_{\mu,2s}(1)=\lim_{x \to 1}\mathrm{A}_{\mu,2s}(x)$ and $\mathrm{B}_{\mu,2s+1}(1)=\lim_{x \to 1}\mathrm{B}_{\mu,2s+1}(x)$ then from \cite[Thm. 3.1]{Dunster:2021:NKF} it follows that (\ref{eq6.29}) and (\ref{eq6.30}) hold uniformly for $0 \leq x \leq 1$. As mentioned earlier, this method can also be applied to the coefficient functions in \cref{thm:nularge,thm.conical,thm:mularge} as an alternative to using Cauchy's integral formula near the pole.
\end{remark}

\begin{proof}
From (\ref{eq2.6})  and (\ref{eq2.12}) 
\begin{equation}
\label{eq5.3}
\xi(x \pm i0) =\pm i \eta,\quad 
\beta(x \pm i0) = \mp i \gamma,
\end{equation}
where $\eta$ and $\gamma$ are given by (\ref{eq5.4}). Next (\ref{eq5.9}) follows from (\ref{eq3.24e}), (\ref{eq5.1}), (\ref{eq5.3}) and \cite[Eq. 10.27.6]{NIST:DLMF}, and likewise (\ref{eq5.10}) is obtained from (\ref{eq3.24ee}), (\ref{eq5.2}), (\ref{eq5.3}) and \cite[Eq. 10.27.10]{NIST:DLMF}.

Finally, (\ref{eq3.24f}) - (\ref{eq3.34a}), (\ref{eq5.3}) and (\ref{eq5.4}) confirm that the coefficient functions possess the asymptotic expansions (\ref{eq5.5}) and (\ref{eq5.6}).
\end{proof}

For the conical function $\mathsf{P}^{-\mu}_{-\frac{1}{2}+i\tau}(x)$ we have the following.

\begin{theorem}
Let be $L_{-\frac{1}{2}+i\tau}^{\mu}$ be given by (\ref{eq3.24j}). Then
\begin{multline}
\label{eq5.11}
\mathsf{P}^{-\mu}_{-\frac{1}{2}+i\tau}(x)
= e^{\mu \pi i/2}L_{-\frac{1}{2}+i\tau}^{\mu} \, 
\Gamma\left(i\tau-\mu+\tfrac12 \right)
\left(1-x^2\right)^{-1/4} 
\\ \times
\eta^{1/2}\left\{I_{\mu}(\tau \eta)
A_{\mu}(i\tau,x)
- i \eta I_{\mu+1}(\tau \eta )
B_{\mu}(i\tau,x)\right\},
\end{multline}
where
\begin{equation}
\label{eq5.12}
A_{\mu}(i\tau,x) \sim 
\exp \left\{
\sum\limits_{s=1}^{\infty}
(-1)^{s}\frac{\tilde{\mathcal{F}}_{\mu,2s}(x)}{\tau^{2s}}
\right\}
\cosh \left\{ \sum\limits_{s=0}^{\infty } (-1)^{s}
\frac{\tilde{i\mathcal{F}}_{\mu,2s+1}(x)}{\tau^{2s+1}}
\right\},
\end{equation}
and
\begin{equation}
\label{eq5.13}
B_{\mu}(i\tau,x) \sim \frac{i}{\eta}
\exp \left\{
\sum\limits_{s=1}^{\infty}
(-1)^{s}\frac{\mathcal{F}_{\mu,2s}(x)}{\tau^{2s}}
\right\}
\sinh \left\{ \sum\limits_{s=0}^{\infty } (-1)^{s}
\frac{i\mathcal{F}_{\mu,2s+1}(x)}{\tau^{2s+1}}
\right\},
\end{equation}
as $\tau \rightarrow \infty$, uniformly for $0 \leq x \leq 1-\delta<1$ and $0 \leq \mu \leq \mu_{0}< \infty$.
\end{theorem}

\begin{proof}
From (\ref{eq5.1}), (\ref{eq9.1b}), (\ref{eq5.3}) and \cite[Eq. 10.27.6]{NIST:DLMF} one obtains (\ref{eq5.11}). The expansions (\ref{eq5.12}) and (\ref{eq5.13}) are derived from (\ref{eq3.24fff}), (\ref{eq3.24ggg}), (\ref{eq5.3}), (\ref{eq5.7}) and (\ref{eq5.8}).
\end{proof}

Again near $x=1$ one can re-expand (\ref{eq5.12}) and (\ref{eq5.13}) in the form similar to (\ref{eq6.29}) and (\ref{eq6.30}) which provides expansions that are valid near and at $x=1$. Note that a similar relation for $\mathsf{Q}^{\mu}_{-\frac{1}{2}+i\tau}(x)$ is obtainable from (\ref{eq5.2}), (\ref{eq9.1c}) and (\ref{eq5.3}).

\subsection{Large \texorpdfstring{$\mu$}{}}

Rather than using the complex-valued expansions as we did above, we obtain expansions in terms of elementary functions (as opposed to modified Bessel functions) directly from the differential equation (\ref{eq1.2x}) with $z=x \in (-1,1)$. Removing the first derivative in a standard manner with the new dependent variable $y(x)=(1-x^2)^{1/2}w(x)$ results in it being recast in the form (see \cite[Eq. (4.1)]{Dunster:2003:ALF})
\begin{equation}
\label{eq5.14}
\frac{d^2 y}{dx^2} = \left\{
\mu^2 \mathrm{f}(\alpha,x) + \mathrm{g}(x)\right\} y,
\end{equation}
where
\begin{equation}
\label{eq5.15}
\mathrm{f}(\alpha,x)
=\frac {1-\alpha^{2} \left(1-x^{2}\right) }
{\left(1-x^{2}\right) ^{2}}, \;
 \mathrm{g}(x) = - \frac{x^{2}+3}
{4\left(1-x^{2}\right) ^{2}}.
\end{equation}
Note that $\mathrm{f}$ and $\mathrm{g}$ differ from $f$ and $g$ as given by (\ref{eq2.4}).

Next let 
\begin{equation}
\label{eq5.16}
\alpha=u/\mu=\left(\nu+\tfrac12\right)/\mu, \;
\tilde{\alpha} = \tau/\mu.
\end{equation}
We consider $\mu \to \infty$ for two cases: (i) $\nu \in \mathbb{R}$ such that $0 \leq \alpha \leq 1-\delta$, and (ii) $\nu = -\frac12 + i \tau$ with $\tau >0$ and $0 < \tilde{\alpha} \leq \tilde{\alpha}_{0}<\infty$ where $\tilde{\alpha}_{0}$ is arbitrary. Thus in both of these $|\nu|$ does not have to be bounded.

For case (i) $\mathrm{f}(\alpha,x)$ has turning points on the imaginary axis at $x=\pm i\alpha^{-1}(1-\alpha^2)^{1/2}$ which are bounded away from $(-1,1)$ for  $0 \leq \alpha \leq 1 -\delta$. For case (ii) we have from (\ref{eq5.16}) that $\alpha$ is replaced by $i\tilde{\alpha}$. Thus in (\ref{eq5.19}) $\mathrm{f}(\alpha,x)$ is replaced by
\begin{equation}
\label{eq5.26}
\mathrm{f}(i\tilde{\alpha},x)
=\frac {\tilde{\alpha}^{2} \left(1-x^{2}\right)+1 }
{\left(1-x^{2}\right) ^{2}}.
\end{equation}

As we mentioned, the ODE (\ref{eq5.14}) follows from (\ref{eq1.2x}) and it has solutions $(1-x^2)^{1/2}\mathsf{P}_{\nu}^{-\mu}(\pm x)$ which we shall approximate.  From \cite[Eqs. 14.8.1 and 14.2.3]{NIST:DLMF}
\begin{equation*}
\label{eq5.16a}
\mathscr{W}\left\{\mathsf{P}_{\nu}^{-\mu}(x),
\mathsf{P}_{\nu}^{-\mu}(-x)\right\}
=\frac{2}{\Gamma(\mu-\nu)\Gamma(\nu+\mu+1)
\left(1-x^2\right)}.
\end{equation*}
Note in both our cases described above $\nu-\mu \neq 0, 1, 2, \ldots$, and $\nu + \mu \neq -1, -2, -3, \ldots$, and thus they are linearly independent. These then form a numerically satisfactory pair of solutions for $x \in (-1,1)$ due to their respective recessive behaviour at the singularity $x= \pm 1$; see (\ref{eq5.2a}).

\begin{theorem}
Let $\alpha=(\nu+\frac12)/\mu$,
\begin{equation}
\label{eq5.20}
p=x \left\{1-\alpha^{2} \left(1-x^{2}\right)\right\}^{-1/2},
\end{equation}
and
\begin{equation}
\label{eq5.17}
\chi = \frac{1}{2}\ln\left(\frac{1+p}{1-p}\right)
+\frac{\alpha}{2}
\ln\left(\frac{1-\alpha p}{1+\alpha p}\right).
\end{equation}
Define 
\begin{equation}
\label{eq5.24}
\mathrm{E}_{s}(\alpha,p) =\left(1-\alpha^{2}\right)
\int_{0}^{p}
\frac{\mathrm{F}_{s}(\alpha,q)}
{\left(1-q^2\right)\left(1-\alpha^2 q^2\right)} d q
\quad (s=1,2,3,\ldots),
\end{equation}
where
\begin{equation}
\label{eq5.21}
\mathrm{F}_{1}(\alpha,p)
=\frac{\left(1-p^2\right)\left(1-\alpha^2 p^2\right)
\left\{\alpha^2\left(1-5 p^2\right)+1\right\}}
{8\left(1-\alpha^2\right)^2},
\end{equation}
\begin{equation}
\label{eq5.22}
\mathrm{F}_{2}(\alpha,p)
=\frac{p\left(1-p^2\right)\left(1-\alpha^2 p^2\right)
\left\{
\alpha^4 \left(15p^4 - 12p^2 + 1 \right)
+\alpha^2 \left(7-12p^2\right)+1
\right\}}
{8\left(1-\alpha^2\right)^3},
\end{equation}
and
\begin{multline}
\label{eq5.23}
\mathrm{F}_{s+1}(\alpha,p)
=\frac{\left(1-p^2\right)\left(\alpha^2 p^2-1\right)}
{2\left(1-\alpha^{2}\right)}
\frac{d\mathrm{F}_{s}(\alpha,p)}{d p}
\\
-{\frac{1}{2}}\sum\limits_{j=1}^{s-1}
\mathrm{F}_{j}(\alpha,p) \mathrm{F}_{s-j}(\alpha,p)
\quad (s=2,3,4,\ldots).
\end{multline}
Then $\mathrm{E}_{s}(\alpha,p)$ are polynomials in $p$, and
\begin{multline}
\label{eq5.25}
\mathsf{P}_{\nu}^{-\mu}(\pm x)
=\frac{\sqrt{\pi}
\left(1-\alpha^{2}\right)^{1/4} p^{1/2}}{2^{\mu}
\Gamma(\tfrac12 \nu +\tfrac12 \mu+1)
\Gamma(\tfrac12 \mu+\tfrac12-\tfrac12 \nu) x^{1/2}}
\\ \times
\exp\left\{\mp \mu \chi 
+\sum\limits_{s=1}^{N-1}
(\mp 1)^{s}\frac{\mathrm{E}_{s}(\alpha,p)}{\mu^{s}}
\right\}
\left\{1+\mathcal{O}
\left(\frac{x}{\mu^{N}}\right)
\right\}
\quad (\mu \to \infty),
\end{multline}
for arbitrary positive integer $N$, uniformly for $-1<x<1$ and $0 \leq \alpha \leq 1- \delta<1$. 

Further, let $\tilde{\alpha}=\tau/\mu$, $\tilde{p}$ be given by (\ref{eq5.20})) with $\alpha^2$ replaced by $-\tilde{\alpha}^2$, and
\begin{equation}
\label{eq5.28}
\tilde{\chi}
=\frac{1}{2}\ln\left(\frac{1+\tilde{p}}
{1-\tilde{p}}\right)
+\frac{\tilde{\alpha}}{2}
\arctan\left(\frac{2\tilde{\alpha} \tilde{p}}
{1-\tilde{\alpha}^2 \tilde{p}^2}\right).
\end{equation}
Then for $N$ again an arbitrary positive integer
\begin{multline}
\label{eq5.27}
\mathsf{P}_{-\frac12+i\tau}^{-\mu}(\pm x)
=\frac{\sqrt{\pi}
\left(1+\tilde{\alpha}^{2}\right)^{1/4} \tilde{p}^{1/2}}
{2^{\mu}|\Gamma(\tfrac12 \mu+\frac34+\frac12 i\tau)|^{2}x^{1/2}}
\\ \times
\exp\left\{\mp \mu \tilde{\chi}
+\sum\limits_{s=1}^{N-1}
(\mp 1)^{s}\frac{\mathrm{E}_{s}
(i\tilde{\alpha},\tilde{p})}{\mu^{s}}
\right\}
\left\{1+\mathcal{O}
\left(\frac{x}{\mu^{N}}\right)
\right\}
\quad (\mu \to \infty),
\end{multline}
uniformly for $-1<x<1$ with $\tilde{\alpha} \in (0,\infty)$ fixed.
\end{theorem}

\begin{proof}
As in \cref{lem.LGconical} the expansions on the RHS of (\ref{eq5.25}) come from the LG asymptotic solutions of (\ref{eq5.14}) provided by \cite[Eqs. (1.1), (1.3), (1.4), (1.6), (1.10) - (1.12), (1.14) - (1.23)]{Dunster:2020:LGE}, with $z$ and $u$ of that reference replaced by $x$ and $\mu$, respectively. On replacing $\xi$ with $\chi$ in \cite[Eq. (1.3)]{Dunster:2020:LGE} (to avoid confusion with (\ref{eq2.13a})) the Liouville variable in the expansions is given by
\begin{equation}
\label{eqchi}
\chi = \int_{0}^{x} \mathrm{f}^{1/2}(\alpha,t) dt.
\end{equation}
The variable $p$ defined by (\ref{eq5.20}) yields simpler expressions (see \cite[Eq. (4.11)]{Dunster:2003:ALF}). Thus on using this and (\ref{eq5.15}) one obtains (\ref{eq5.17}) upon explicit integration. 

The coefficients in the LG expansions are given by \cite[Eqs. (1.10) - (1.12)]{Dunster:2020:LGE}, with the first two intermediary coefficients given in general by 
\begin{equation}
\label{eq5.19}
\mathrm{F}_{1}
=\tfrac12 \psi, \;
\mathrm{F}_{2}
=-\tfrac14 d\psi/d\chi.
\end{equation}
From (\ref{eq5.15}) and \cite[Eq. (1.6)]{Dunster:2020:LGE} the Schwarzian derivative $\psi=\psi(\alpha,x)$ here is explicitly given by
\begin{multline}
\label{eq5.18}
\psi(\alpha,x) =\frac{4 \mathrm{f}(\alpha,x) 
\mathrm{f}^{\prime \prime }(\alpha,x) 
-5 \mathrm{f}^{\prime 2}(\alpha,x) }
{16 \mathrm{f}^{3}(\alpha,x) }
+\frac{\mathrm{g}(x) }{\mathrm{f}(\alpha,x) }
\\
=\frac{\left(1-x^2\right)
\left\{1- 4\alpha^2 x^2-\alpha^4 (1-x^2 ) \right\}}
{4\left\{1 - \alpha^2\left(1-x^2\right) \right\}^3},
\end{multline}
with primes denoting derivatives with respect to $x$. This is analytic at $x=\pm 1$; this is a consequence of (\ref{eq5.15}) and \cite[Chap. 10, Thm. 4.1]{Olver:1997:ASF}. Now from (\ref{eq5.15}) and (\ref{eq5.17})
\begin{equation}
\label{eq5.20a}
\frac{dp}{d\chi}
=\frac{\left(1-p^2\right)\left(1-\alpha^2 p^2\right)}
{1-\alpha^{2}}.
\end{equation}
Consequently from (\ref{eq5.20}), (\ref{eq5.17}) and (\ref{eq5.19}) - (\ref{eq5.20a}) we obtain the coefficients (\ref{eq5.21}) and (\ref{eq5.22}). The subsequent $\mathrm{F}_{s}$ coefficients come from (\ref{eq5.20a}) and \cite[Eq. (1.12)]{Dunster:2020:LGE}, and in terms of the variable $p$ are the polynomials given recursively by (\ref{eq5.23}). By induction it is readily verifiable that they all have the factor  $(1-p^2)(\alpha^2 p^2-1)$. Thus the coefficients $\mathrm{E}_{s}(\alpha,p)$ are then also evidently polynomials in $p$ as given by (\ref{eq5.24}), which itself follows from (\ref{eq5.20a}) and \cite[Eq. (1.10)]{Dunster:2020:LGE}.

Next on matching solutions that are recessive at $x=\pm 1$ one arrives at (\ref{eq5.25}). The proportionality constants here come from matching solutions at $x=0$ and referring to \cite[Eq. 14.5.1]{NIST:DLMF}, which also gives the term $x$ appearing in the order term in (\ref{eq5.25}). 

Finally, when $\nu = -\frac12 + i \tau$ the only change is that $\mathrm{f}(\alpha,x)$ is replaced by $\mathrm{f}(i\tilde{\alpha},x)$ as given by (\ref{eq5.26}). Hence in (\ref{eq5.20}) - (\ref{eq5.25}) we replace $\alpha^2$ by $-\tilde{\alpha}^2$. Moreover, equation (\ref{eq5.14}) now has turning points at $x=\pm \tilde{\alpha}^{-1}(1+\tilde{\alpha}^2)^{1/2}$ which are bounded away from $(-1,1)$ for $\tau >0$ and $0 < \tilde{\alpha} \leq \tilde{\alpha}_{0}<\infty$ where $\tilde{\alpha}_{0}$ is arbitrary. This confirms that (\ref{eq5.27}) is valid for this parameter range.
\end{proof}

\section{Numerical results}
\label{sec.numerics}

From \cite[Eqs. 14.2.11, 14.3.10 and 14.9.14]{NIST:DLMF} the functions $R_{\nu}^{\mu}(z)$, $\mathrm{R}_{\nu}^{\mu}(x)$ and $\mathrm{S}_{\nu}^{\mu}(x)$ are all identically equal to $1$, where for $z \in \mathbb{C}$
\begin{equation}
\label{eq6.1}
R_{\nu}^{\mu}(z)
=\Gamma(\nu+\mu+2)\left\{
(\mu-\nu+1)P_{\nu}^{-\mu}(z)
\boldsymbol{Q}_{\nu+1}^{\mu}(z)
+P_{\nu+1}^{-\mu}(z)\boldsymbol{Q}_{\nu}^{\mu}(z)
\right\},
\end{equation}
and for $x \in (-1,1)$
\begin{equation}
\label{eq6.2}
\mathrm{R}_{\nu}^{\mu}(x)
=\frac{\Gamma(\nu+\mu+2)}{\Gamma(\nu-\mu+1)}
\left\{\mathsf{P}_{\nu+1}^{-\mu}(x)
\mathsf{Q}_{\nu}^{-\mu}(x)
-\mathsf{P}_{\nu}^{-\mu}(x)
\mathsf{Q}_{\nu+1}^{-\mu}(x)\right\},
\end{equation}
\begin{equation}
\label{eq6.3}
\mathrm{S}_{\nu}^{\mu}(x)
=\tfrac12 \Gamma(\nu+\mu+2)\Gamma(\mu-\nu)
\left\{\mathsf{P}_{\nu+1}^{-\mu}(x)
\mathsf{P}_{\nu}^{-\mu}(-x)
+\mathsf{P}_{\nu}^{-\mu}(x)
\mathsf{P}_{\nu+1}^{-\mu}(-x)\right\}.
\end{equation}

We shall insert our truncated asymptotic approximations for the associated Legendre and Ferrers functions in all three of these to determine how close to 1 they are for some chosen values of the parameters. 

Starting with (\ref{eq6.1}), and based on \cref{thm:nularge}, for positive integer $N$ define
\begin{equation}
\label{eq6.4}
A_{\mu}(N,u,\beta) =
\exp \left\{
\sum\limits_{s=1}^{N}
\frac{\tilde{\mathcal{E}}_{\mu,2s}(z)}{u^{2s}}
\right\}
\cosh \left\{ \sum\limits_{s=0}^{N}
\frac{\tilde{\mathcal{E}}_{\mu,2s+1}(z)}{u^{2s+1}}
\right\},
\end{equation}
and
\begin{equation}
\label{eq6.5}
B_{\mu}(N,u,\beta) = 
\frac{1}{\xi}\exp \left\{
\sum\limits_{s=1}^{N}
\frac{\mathcal{E}_{\mu,2s}(z)}{u^{2s}}
\right\}
\sinh \left\{ \sum\limits_{s=0}^{N}
\frac{\mathcal{E}_{\mu,2s+1}(z)}{u^{2s+1}}
\right\},
\end{equation}
where here and in (\ref{eq3.34}) and (\ref{eq3.34a}) $\xi$ is given in terms of $\beta$ by (\ref{eq2.13a}). Then from (\ref{eq3.24e}) and (\ref{eq3.24ee}) for large $u$
\begin{multline}
\label{eq6.6}
P_{\nu}^{-\mu}(z) \approx I_{\mu}(N,u,\beta)
:= L_{u-\frac12}^{\mu} \, \Gamma(u-\mu+\tfrac12)
 \left(\beta^2-1\right)^{1/4}
\\ \times
 \xi^{1/2} \left\{I_{\mu}(u \xi)A_{\mu}(N,u,\beta)
+\xi I_{\mu+1}(u \xi)B_{\mu}(N,u,\beta)\right\},
\end{multline}
and
\begin{multline}
\label{eq6.7}
\boldsymbol{Q}^{\mu}_{\nu}(z) 
\approx K_{\mu}(N,u,\beta)
:= L_{u-\frac12}^{\mu} \, 
\left(\beta^2-1\right)^{1/4} 
\\ \times
\xi^{1/2}
\left\{K_{\mu}(u \xi)A_{\mu}(N,u,\beta)
-\xi K_{\mu+1}(u \xi)B_{\mu}(N,u,\beta)\right\}.
\end{multline}

From (\ref{eq6.1}) we test how close the approximation $R_{\mu}(N,u,\beta)$ is to the value 1, where on recalling $u=\nu+\frac12$,
\begin{multline}
\label{eq6.8}
R_{\mu}(N,u,\beta)
=\Gamma\left(u+\mu+\tfrac32 \right) 
\\ \times
\left\{
(\mu-u+\tfrac32)I_{\mu}(N,u,\beta)
K_{\mu}(N,u+1,\beta)
+I_{\mu}(N,u+1,\beta)K_{\mu}(N,u,\beta)
\right\}.
\end{multline}
Thus let
\begin{equation}
\label{eq6.8a}
\Delta_{\mu}(N,u,\beta)
=R_{\mu}(N,u,\beta)-1,
\end{equation}
and consider the absolue value of this function along the semi-circle, line segment, and unbounded line in the $z$ plane, as depicted in \cref{fig:Fig1} connecting the points labelled $\mathsf{A}$ $\mathsf{B}$, $\mathsf{C}$ and $\mathsf{D}$; the corresponding curves in the $\beta$ plane are shown in \cref{fig:Fig2}. Here the point labelled $\mathsf{D}$ in the $z$ plane is assumed to be arbitrarily large along the positive imaginary axis, and correspondingly in the $\beta$ plane the point is arbitrarily close to $\beta=1$. 

\begin{figure}[H]
 \centering
 \includegraphics[
 width=0.9\textwidth,keepaspectratio]{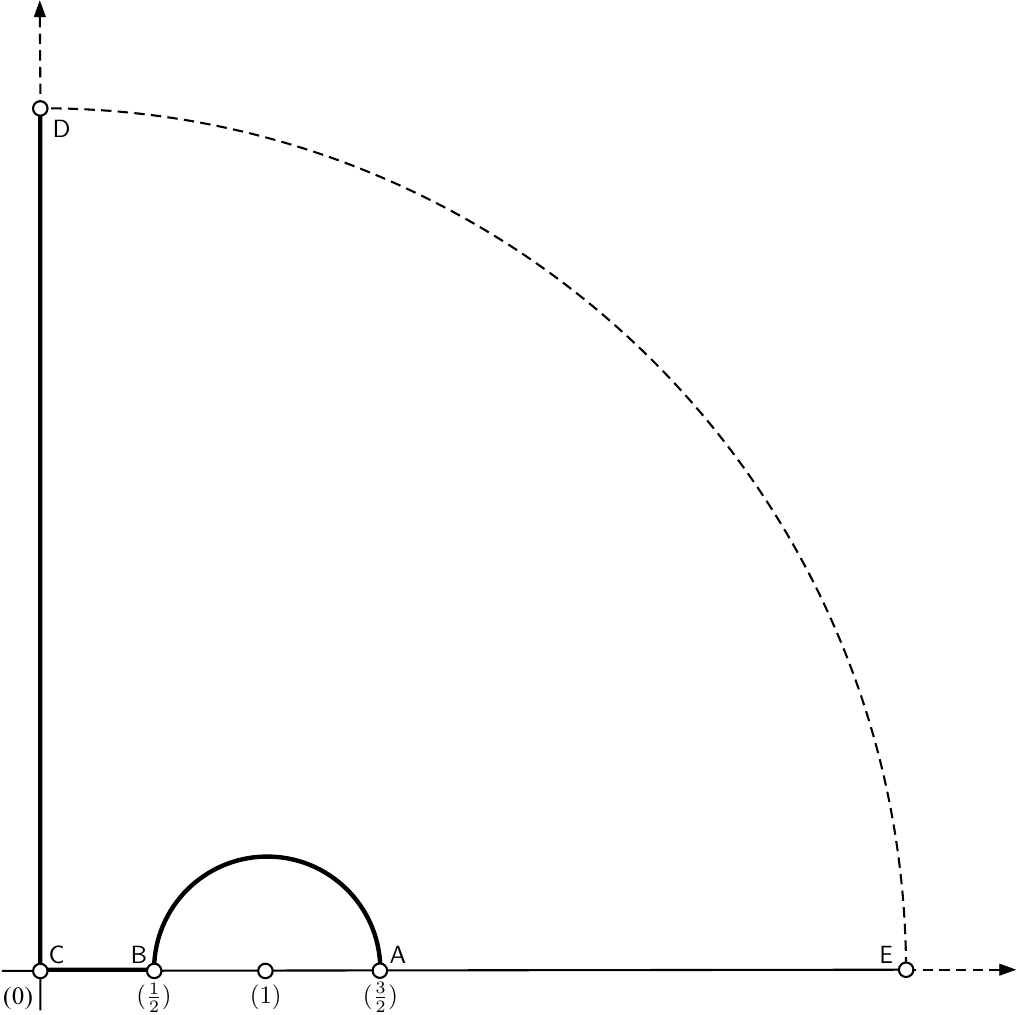}
 \caption{$z$ plane.}
 \label{fig:Fig1}
\end{figure}

We choose the values $u=20.8$, $\mu=4.2$ and $N=11$, and in \cref{fig:Fig3,fig:Fig4,fig:Fig5} graphs of $|\Delta_{4.2}(11,20.8,\beta)|$ along these three curves are shown. From these we computed that 
\begin{equation}
\label{eq6.8b}
|\Delta_{4.2}(11,20.8,\beta)|
\leq 1.18724 \cdots
\times 10^{-16},
\end{equation}
with the maximum attained on the curve $\mathsf{A}\mathsf{B}$ at the point corresponding to $z=1+0.5e^{i\theta}$ with $\theta = 0.90632\cdots$ ($z \approx 1.30832+0.39362 \,i$, $\beta \approx 1.22484 - 0.30671 \,i$). From the maximum modulus theorem, and the Schwarz reflection principle, we conclude that the bound (\ref{eq6.8b}) holds for $|z-1| \geq \frac12$ with $|\arg(z)|\leq \frac12 \pi$.

\begin{figure}
 \centering
 \includegraphics[
 width=0.7\textwidth,keepaspectratio]{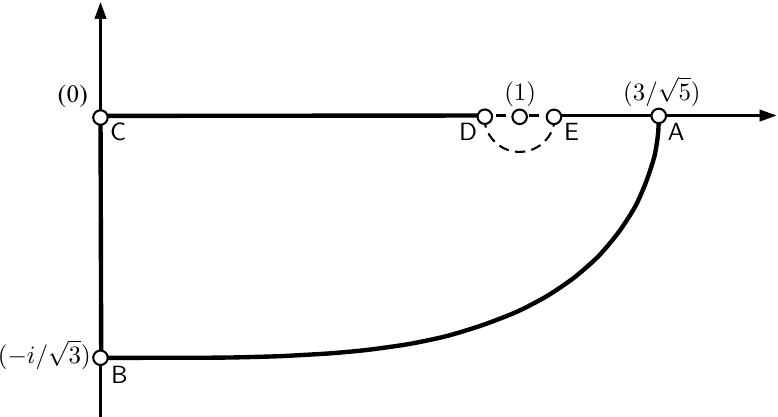}
 \caption{$\beta$ plane.}
 \label{fig:Fig2}
\end{figure}

\begin{figure}
 \centering
 \includegraphics[
 width=0.7\textwidth,keepaspectratio]{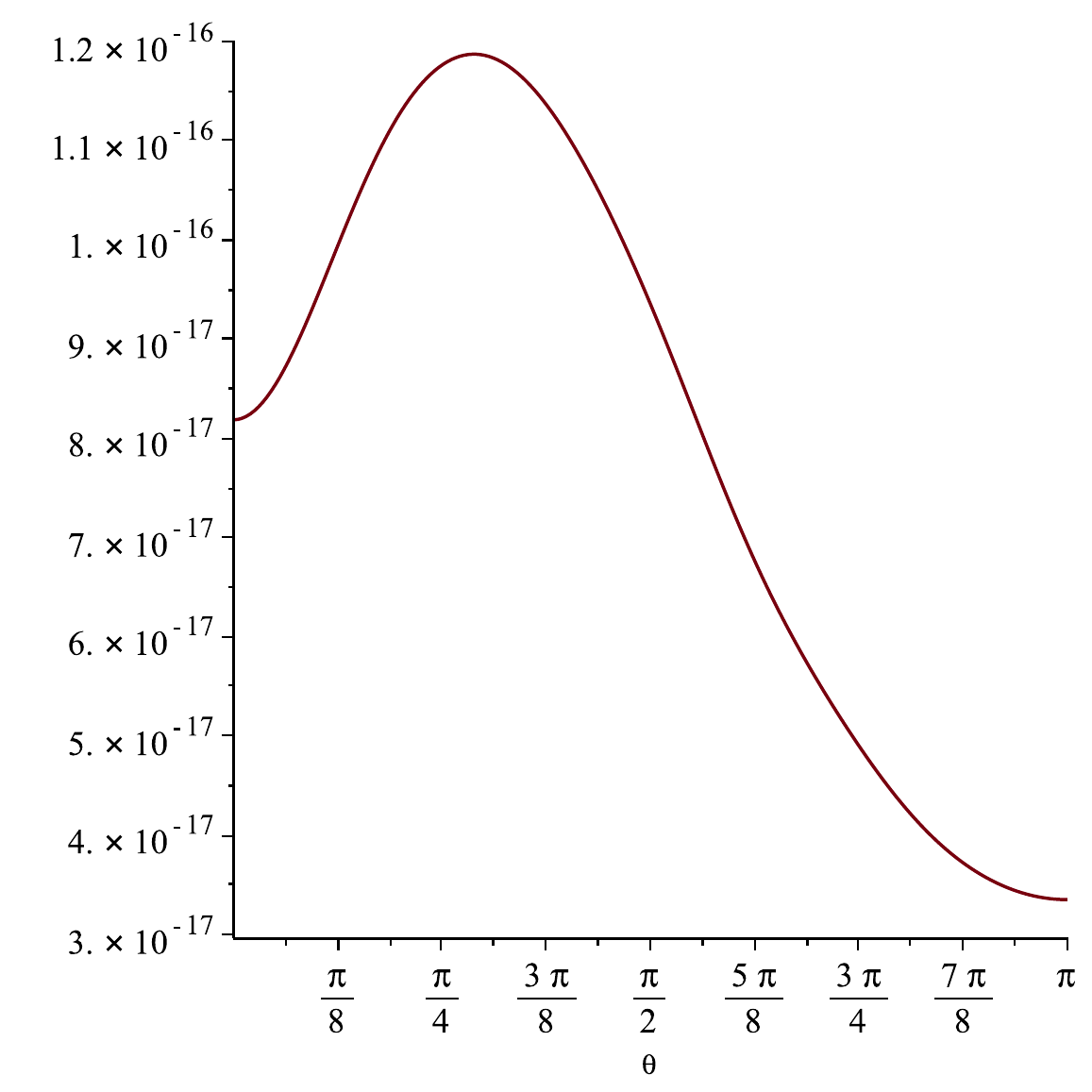}
 \caption{Graph of $|\Delta_{4.2}(11,20.8,\beta)|$ for $\beta$ lying on the curve $\mathsf{A}\mathsf{B}$ corresponding to $z=1+0.5e^{i\theta}$, $0 \leq \theta \leq \pi$.}
 \label{fig:Fig3}
\end{figure}

\begin{figure}
 \centering
 \includegraphics[
 width=0.7\textwidth,keepaspectratio]{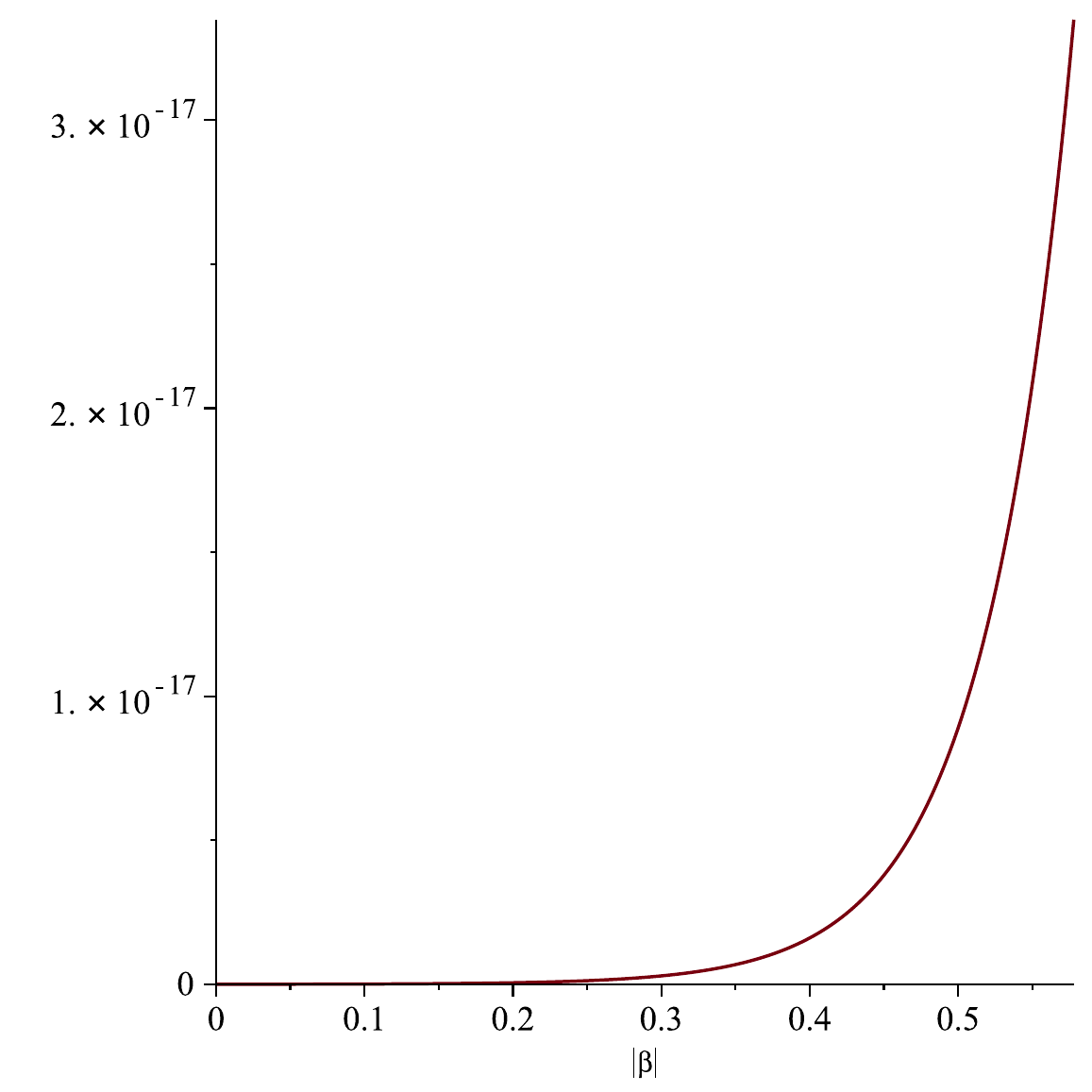}
 \caption{Graph of $|\Delta_{4.2}(11,20.8,\beta)|$ for $\beta=-iy$, $0\leq y \leq 1/\sqrt{3}$.}
 \label{fig:Fig4}
\end{figure}

\begin{figure}
 \centering
 \includegraphics[
 width=0.7\textwidth,keepaspectratio]{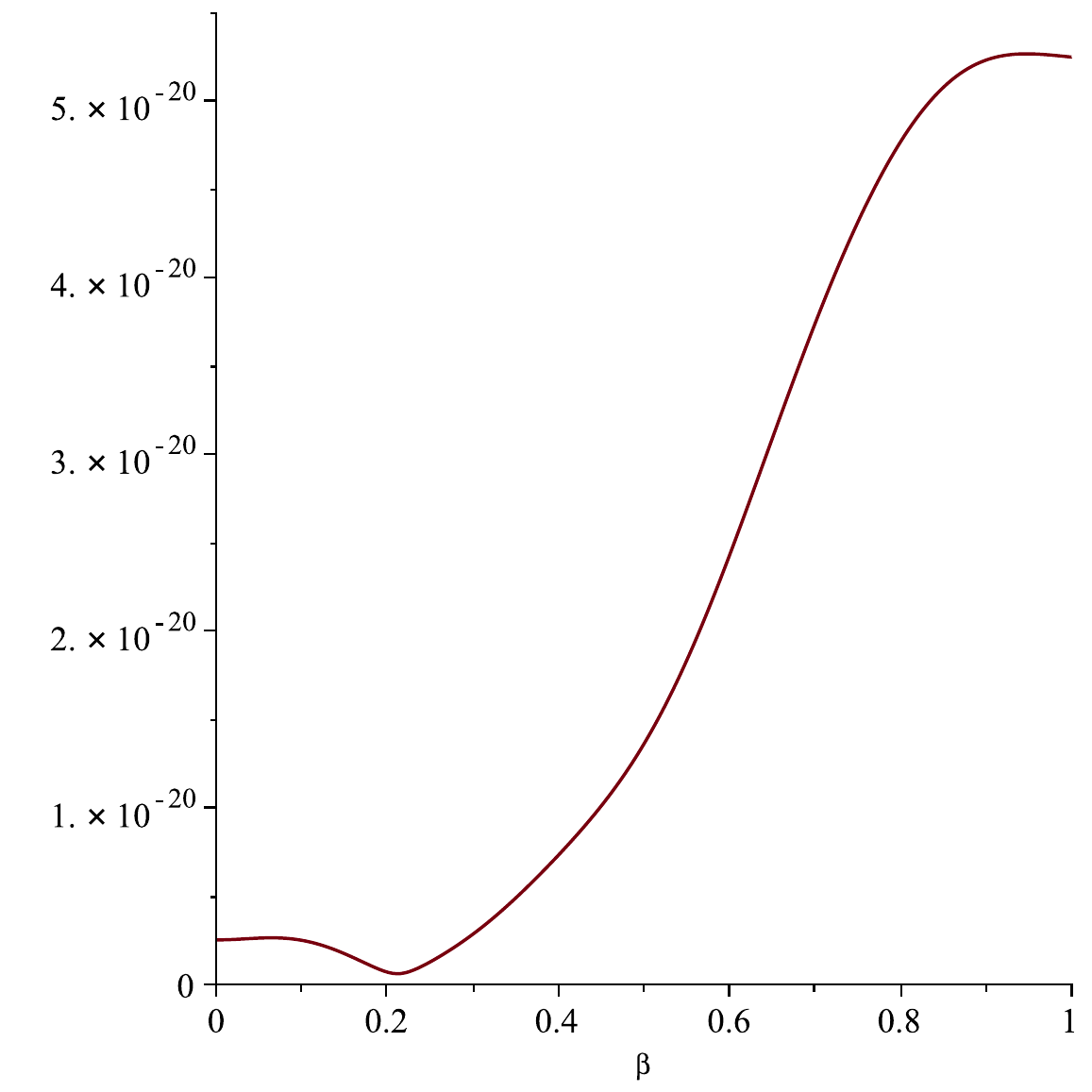}
 \caption{Graph of $|\Delta_{4.2}(11,20.8,\beta)|$ for $0\leq \beta \leq 1$.}
 \label{fig:Fig5}
\end{figure}

For $|z-1| \leq \frac12$ we return to using $z$ instead of $\beta$ and test the identity
\begin{equation}
\label{eq6.20}
R_{\nu}^{\mu}(z)
=\frac{1}{2\pi i}\oint_{\mathscr{C}} 
\frac{R_{\nu}^{\mu}(t)}{t-z}dt = 1,
\end{equation}
where $\mathscr{C}$ must enclose $z$ and lie in the half-plane $|\arg(t)|\leq \frac12 \pi$, which we choose to be the unit circle centred at $t=1$, orientated positively.

We again replace $P$ and $\boldsymbol{Q}$ in (\ref{eq6.1}) by their approximations $I_{\mu}(N,u,\beta)$ and $K_{\mu}(N,u,\beta)$, these being given by (\ref{eq6.6}) and (\ref{eq6.7}), and we regard this as a function of $z$ rather than $\beta$ (see (\ref{eq2.12})). Let us denote this approximation by $\hat{R}_{\mu}(N,u,z)$. Then from (\ref{eq6.20}) consider the following approximation to $R_{\nu}^{\mu}(z)$
\begin{equation}
\label{eq6.24}
\frac{1}{2\pi i}\oint_{\mathscr{C}} 
\frac{\hat{R}_{\mu}(N,u,t)}{t-z}dt = 
1+\hat{\delta}_{\mu}(N,u,z),
\end{equation}
where
\begin{equation}
\label{eq6.25}
\hat{\delta}_{\mu}(N,u,z)
=\frac{1}{2\pi i}\oint_{\mathscr{C}} 
\frac{\hat{R}_{\mu}(N,u,t)-1}{t-z}dt.
\end{equation}
In order to bound this error term parameterise $t=1+e^{i\theta}$ ($-\pi \leq \theta \leq \pi$) and then define
\begin{equation}
\label{eq6.22}
\hat{\Delta}_{\mu}(N,u,\theta)
=\hat{R}_{\mu}\left(N,u,1+e^{i \theta}\right)-1.
\end{equation}
We seek the maximum of the absolute value of this, and in doing so only need to consider $0 \leq \theta \leq \pi$ by virtue of Schwarz symmetry. Accordingly, again with $u=20.8$, $\mu=4.2$ and $N=11$, we find
\begin{equation}
\label{eq6.23}
|\hat{\Delta}_{4.2}(11,20.8,\theta)|
\leq 1.305412279 \cdots \times 10^{-19}
\quad (-\pi \leq \theta \leq \pi),
\end{equation}
with the maximum achieved at $\theta=\pm 0.47449\cdots$; see \cref{fig:Fig6}.

\begin{figure}
 \centering
 \includegraphics[
 width=0.7\textwidth,keepaspectratio]{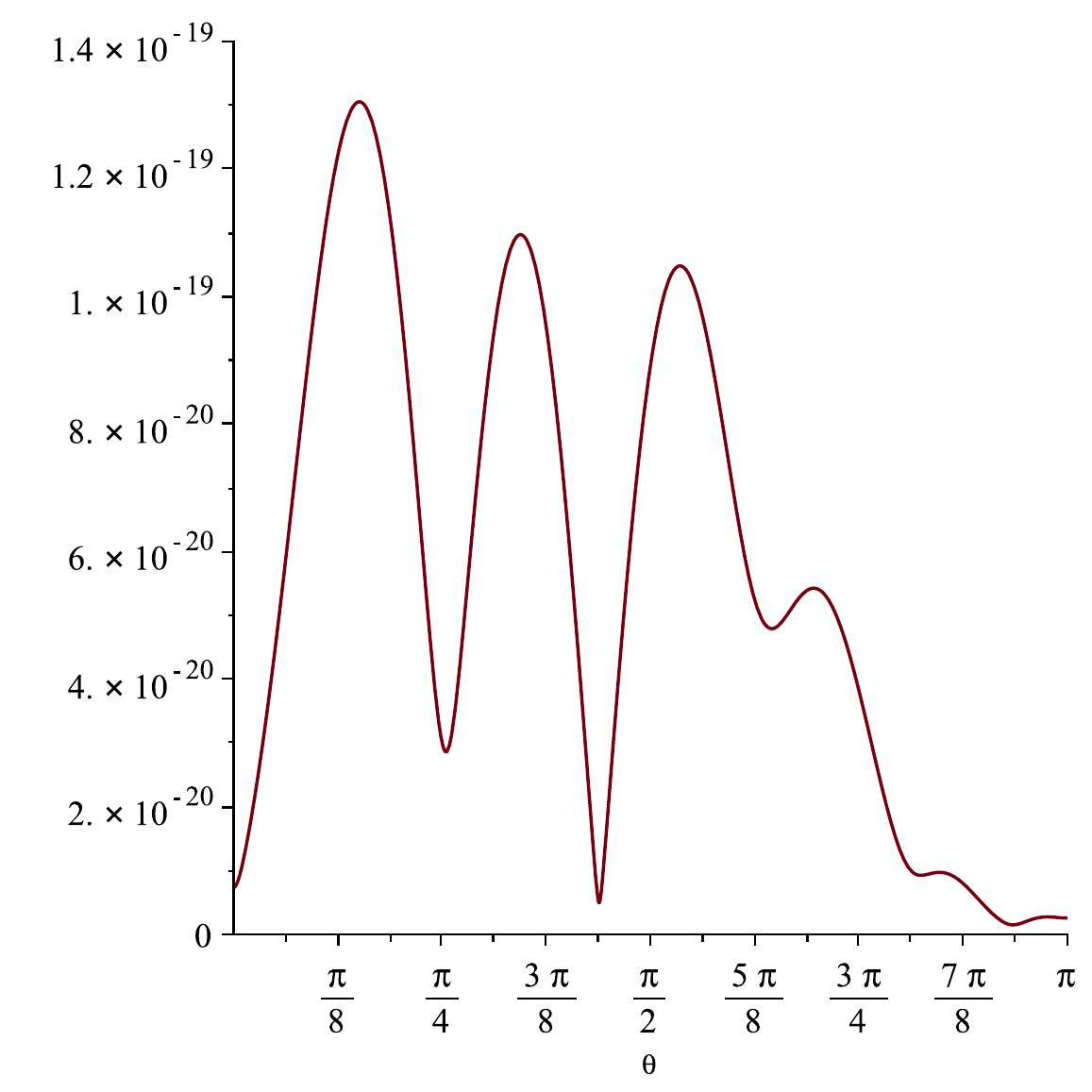}
 \caption{Graph of $|\hat{\Delta}_{4.2}(11,20.8,\theta)|$ for  $0\leq \theta \leq \pi$.}
 \label{fig:Fig6}
\end{figure}

It then follows from (\ref{eq6.22}), (\ref{eq6.23}) and (\ref{eq6.25}) for $u=20.8$, $\mu=4.2$, $N=11$ and $|z-1| \leq \frac12$
\begin{equation}
\label{eq6.26}
|\hat{\delta}_{4.2}(11,20,z)|
< \frac{1.30541228 \times 10^{-19} l_{0}(z)}{2\pi},
\end{equation}
where
\begin{equation}
\label{eq6.27}
l_{0}(z)=
\oint_{\mathscr{C}} 
\left|\frac{dt}{t-z}\right|
=\int_{-\pi}^{\pi}
\frac{d\theta}{\left|e^{i\theta}+1-z\right|}.
\end{equation}

This can be expressed as an elliptic integral, and analysed accordingly (see \cite[Lemma 4.1]{Dunster:2021:SEB}), but here we have explicit values for the parameters in the cited lemma and hence we take a simpler approach. It is straightforward to show that $l_{0}(z)$ is constant along any circle centred at $z=1$, so to maximise it in the disk $|z-1| \leq \frac12$ it suffices to consider $z=x \in [\frac12,\frac32]$. Then we find numerically that the maximum is attained at both end points of this interval, with the value at both being $6.74300141925 \cdots$. Thus from this value inserted into (\ref{eq6.26}) and from (\ref{eq6.20}) we arrive at our desired bound
\begin{equation}
\label{eq6.28}
\left|\frac{1}{2\pi i}\oint_{\mathscr{C}} 
\frac{\hat{R}_{4.2}(11,20,t)}{t-z}dt-1\right|
< 1.40095 \times 10^{-19}
\quad (|z-1| \leq \tfrac12).
\end{equation}

Consider next $\mathrm{R}_{\nu}^{\mu}(x)$ given by (\ref{eq6.2}), which is also identically equal to 1. First we re-expand (\ref{eq5.5}) and (\ref{eq5.6}) as in (\ref{eq6.29}) and (\ref{eq6.30}), taking $N$ terms in both sums, where as noted the coefficients have a removable singularity at $x=1$. Next in (\ref{eq5.9}) and (\ref{eq5.10}) we replace $A_{\mu}(u,x)$ and $B_{\mu}(u,x)$ by these approximations, and then insert these into (\ref{eq6.2}). If we denote this by $\mathrm{R}_{\mu}(N,u,x)$ we find for $u=20.8$, $\mu=4.2$ and $N=5$ that (see \cref{fig:Fig7})

\begin{equation}
\label{eq6.31}
\sup_{0\leq x \leq 1}|\mathrm{R}_{4.2}(5,20.8,x)-1|
= |\mathrm{R}_{4.2}(5,20.8,1)-1|
=4.626048 \cdots \times 10^{-11}.
\end{equation}

Finally, to check an approximation to $\mathrm{S}_{\nu}^{\mu}(x)$ given by (\ref{eq6.3}) (which again is identically equal to 1), we set the $\mathcal{O}(x\mu^{-N})$ term in (\ref{eq5.25}) to zero and insert these approximations into $\mathrm{S}_{\nu}^{\mu}(x)$, and denote this by $\mathrm{S}_{\nu}^{\mu}(N,x)$. We compute for $\nu=4.8$, $\mu=20.3$ and $N=10$ that
\begin{equation}
\label{eq6.32}
\sup_{0\leq x \leq 1}|\mathrm{S}_{4.8}^{20.3}(10,x)-1|
=9.884448 \cdots \times 10^{-12},
\end{equation}
with the supremum attained at $x = 0.331819 \cdots$ (see \cref{fig:Fig8}).

\begin{figure}[H]
 \centering
 \includegraphics[
 width=0.7\textwidth,keepaspectratio]{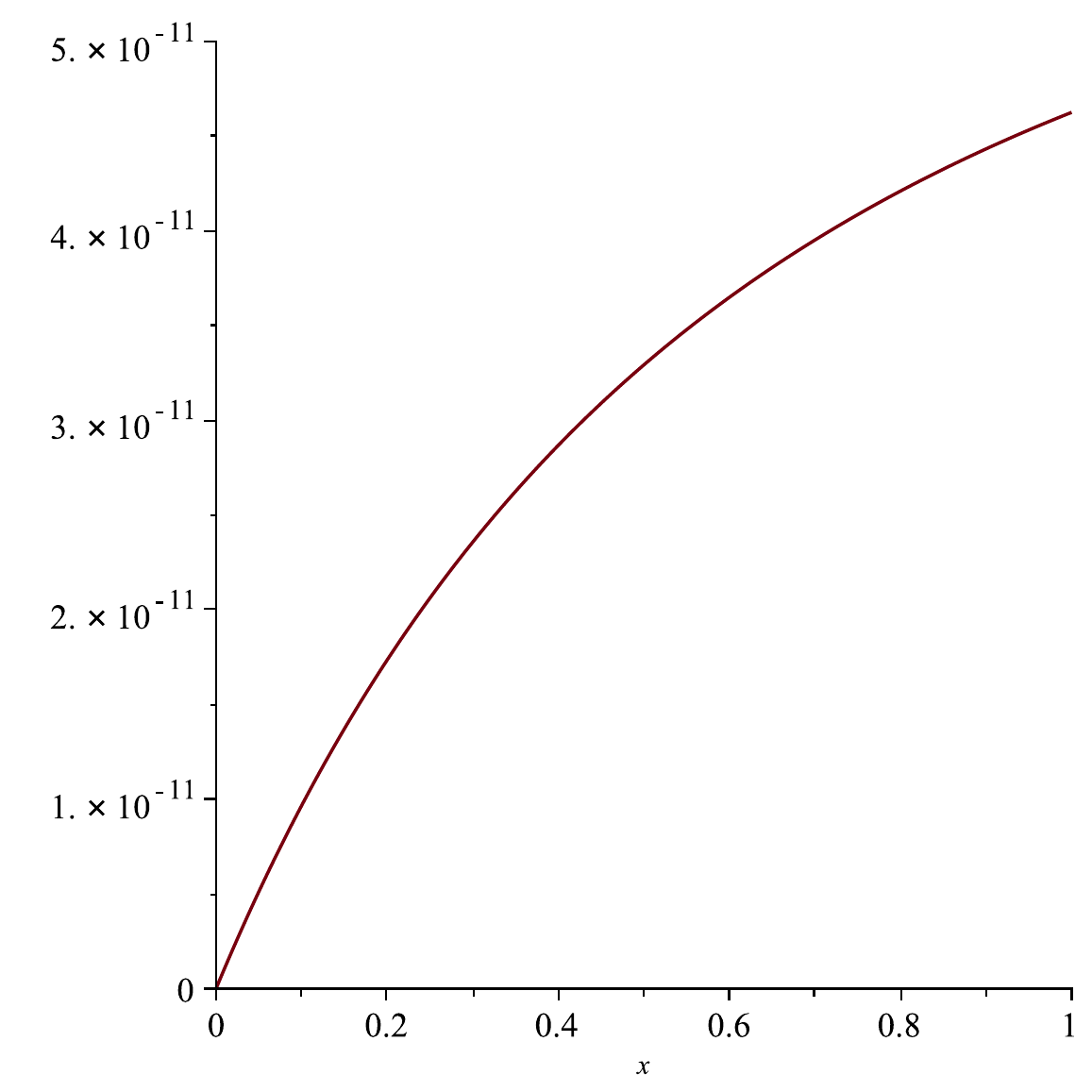}
 \caption{Graph of $|\mathrm{R}_{4.2}(5,20.8,x)-1|$ for  $0\leq x \leq 1$.}
 \label{fig:Fig7}
\end{figure}

\begin{figure}[H]
\centering
\includegraphics[
width=0.7\textwidth,keepaspectratio]{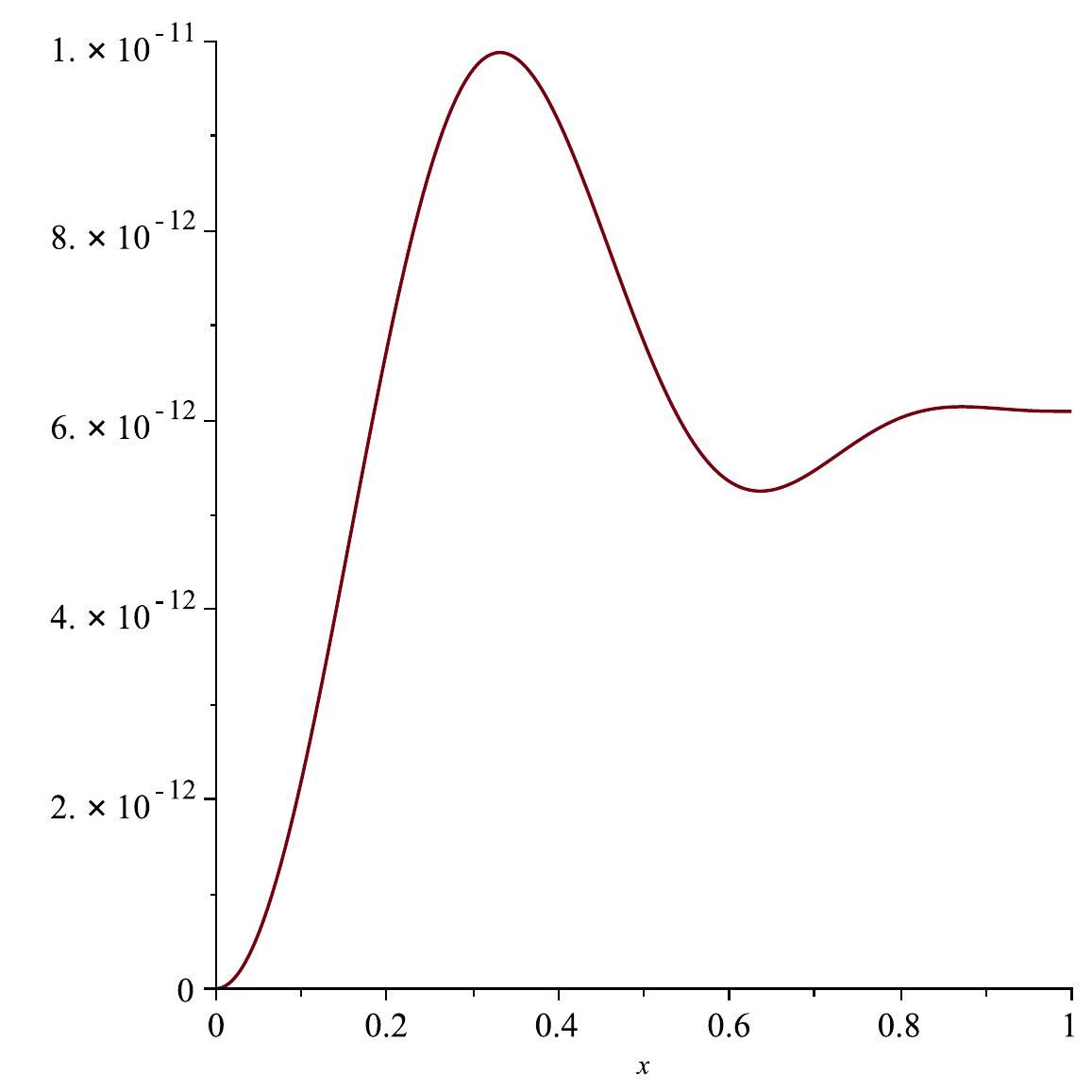}
\caption{Graph of $|\mathrm{S}_{4.8}^{20.3}(10,x)-1|$ for $0\leq x \leq 1$.}
\label{fig:Fig8}
\end{figure}

\section*{Acknowledgments}
I thank the anonymous referees for helpful comments. Financial support from Ministerio de Ciencia e Innovaci\'on, Spain, project PID2021-127252-NB-I00 (MCIU/AEI/FEDER, UE) is acknowledged.

\bibliographystyle{siamplain}
\bibliography{biblio}

\end{document}